\documentclass[10pt]{article}

\usepackage{dcolumn}
 \newcolumntype{d}[1]{D{:=}{\cdot}{#1} }

\usepackage[colorlinks=true, pdfstartview=FitV, linkcolor=blue,
  citecolor=blue, urlcolor=blue]{hyperref}

\usepackage{hyperref}
\usepackage{amsfonts,amsmath, amssymb,latexsym,mathrsfs,upgreek}
\usepackage{multirow}
\usepackage{verbatim}
\usepackage{longtable}
\usepackage{hyperref}
\hypersetup{colorlinks=true,linkcolor=blue, linktocpage}
\usepackage{tocloft}
\usepackage{cancel}

\usepackage{tikz}
\usetikzlibrary{matrix,arrows,decorations.pathreplacing}

\usepackage{todonotes}

\usepackage{amssymb,amsmath, amscd,amsthm}
\usepackage{graphicx}

\usepackage{xurl}

\newtheorem{theoremx}{Theorem}
\newtheorem{theorem}{Theorem}[section]

\newtheorem{corollaryx}[theoremx]{Corollary}
\newtheorem{lemma}[theorem]{Lemma}

\newtheorem{proposition}[theorem]{Proposition}

\def\D{{\mathfrak D}}
\def\p{{\mathfrak p}}

\def\A{{\mathbb A}}

\DeclareMathOperator{\Hom}{Hom}

\DeclareMathOperator{\Aut}{Aut}

\DeclareMathOperator{\SL}{SL}

\DeclareMathOperator{\GL}{GL}

\DeclareMathOperator{\disc}{disc}
\DeclareMathOperator{\Res}{Res}

\DeclareMathOperator{\Sym}{Sym}

\DeclareMathOperator{\Stab}{Stab}

\def\uG{{\underline{G}}}
\def\uV{{\underline{V}}}

\def\disc{{\rm disc}}
\def\Disc{{\rm Disc }}
\def\dim{{\rm dim}}
\def\Aut{{\rm Aut}}

\def\gen{{\rm gen}}
\def\ngen{{\rm ngen}}

\def\Vol{{\rm Vol}}
\def\R{{\mathbb R}}
\def\F{{\mathbb F}}

\def\FF{{\mathcal F}}
\def\PP{{\mathcal P}}

\def\Q{{\mathbb Q}}

\def\K{{F}}

\def\J{{\mathcal J}}
\def\C{{\mathcal C}}

\def\W{{\mathcal W}}

\def\Z{{\mathbb Z}}
\def\P{{\mathbb P}}
\def\F{{\mathbb F}}
\def\FF{{\mathcal F}}
\def\Q{{\mathbb Q}}
\def\C{{\mathbb C}}
\def\A{{\mathbb A}}
\def\GG{{\mathbb G}}

\def\L{{\mathcal L}}

\def\max{{\rm max}}

\def\Cl{{\rm Cl}}
\def\var{{\rm Var}}

\newcommand{\w}[1]{\widetilde{#1}}

\def\cO{{\mathcal O}}
\def\cl{{\rm cl}}
\def\fP{{\mathfrak p}}
\def\sc{{\mathcal C}}
\def\bG{{\mathbb G}}

\DeclareMathOperator{\proj}{proj}
\DeclareMathOperator{\diag}{diag}

\title{Geometry-of-numbers methods over global fields I:
  Prehomogeneous vector spaces} \author{Manjul Bhargava, Arul Shankar,
  and Xiaoheng Wang}

\begin{document}
\maketitle
\begin{abstract}We develop geometry-of-numbers methods to count orbits
  in prehomogeneous vector spaces having bounded invariants over any
  global field.   As our primary example, we apply these techniques to
  determine the density of discriminants of field extensions of degree
  at most 5 over any base global field $F$.
\end{abstract}

\section{Introduction}\label{sec:intro}

In recent years, there have been a number of problems in arithmetic
statistics that have been solved by developing suitable
geometry-of-numbers techniques to count integral orbits of a
representation of an algebraic group on a vector space defined over
$\Z$.  These methods have been applied, e.g., to determine the
densities of discriminants of number field extensions of $\Q$ of small
degree, to establish new cases of the Cohen--Lenstra--Martinet
heuristics for class groups, to bound the average rank of elliptic
curves over $\Q$, and to show that most hyperelliptic curves over $\Q$
have few rational points (see, e.g.,
\cite{hasse}--\cite{BV2}, \cite{KL}, \cite{PS2}, \cite{SW}, \cite{Th} for these and
further examples).  The representations that have so far been the most
useful for the purpose fall into two categories, namely,
prehomogeneous representations having one relative invariant, and
coregular representations having at least two relative invariants,
with the former parametrizing primarily field extensions and related
objects, and the latter parametrizing primarily algebraic curves
together with additional data.

In this series of two articles, our purpose is to generalize the
aforementioned counting methods from the geometry-of-numbers so that
they may be applied over an arbitrary global number or function field.
In this first article, we concentrate on the case of counting integral
orbits in a prehomogeneous vector space.  Our primary application is
the determination of the density
of discriminants of field extensions of degree less than or equal to 5
of any given global field.

Let $F$ denote a global number field or function field. A fundamental
problem in arithmetic statistics is to determine the density of
discriminants of field extensions of $F$ having a fixed degree $n$
over $F$.  The case of $n=2$ when $F=\Q$ (and the case of $n=1$ for
all $F$!) are trivial.  When $n=3$, even the case $F=\Q$ is highly
nontrivial and is a celebrated result of Davenport--Heilbronn
\cite{DH}.  Surprisingly, even the case $n=2$ is nontrivial for more
general $F$ and is a result of Datskovsky--Wright \cite{DW}, who also
settle the case of $n=3$ for general global fields $F$ having
characteristic not $2$ or $3$. The cases $n=4$ and $5$ for $F=\Q$ were
carried out in \cite{dodqf,dodpf}.

If $L$ is any field extension of $F$ of finite degree, then the norm of the relative discriminant of $L$ over $F$ is given by the
local product
\begin{equation}\label{eqnormdisc}
N(\Disc(L/F)):=\prod_{\fP\notin M_\infty}\prod_{w\mid \fP} |\Disc(L_w/F_\fP)|_\fP^{-1},
\end{equation}
where $M_\infty$ denotes the set of archimedean places of $F$,
$|\cdot|_\fP$ denotes the normalized absolute value at~$\fP$, and
$F_\fP$ and $L_w$ denote the completion of $F$ at $\fP$
and the completion of $L$ at $w$, respectively. When $F$ is a function field,
$M_\infty$ is empty.

Our main theorem determines the number of isomorphism classes of field
extensions $L$ of $F$ of degree at most 5 having bounded relative
discriminant:

\begin{theoremx}\label{thmainfield}
  Let $F$ be a global number or function field of any characteristic. For any $n=2$, $3$, $4$ or $5$, we let $N_n(F,X)$ denote the number
  of isomorphism classes of degree-$n$ field extensions $L$ of $F$
  weighted by $(\#\Aut(L/F))^{-1}$, whose normal closure over $F$ has
  full Galois group $S_n$ such that the norm of the relative
  discriminant of $L$ over $F$ is at most $X$ if $F$ is a number
  field and is equal to~$X$ if $F$ is a function field. Let $q(n,k)$
  denote the number of partitions of~$n$ into at most $k$ parts.
\begin{itemize}
\item[{\rm (a)}] If $F$ is a number field with $r_1$ real embeddings
  and $2r_2$ complex embeddings, then
\begin{equation}\label{eq:numberfieldcount}
\lim_{X\rightarrow\infty}\frac{N_n(F,X)}{X} = \frac12\,
\underset{s=1}{\operatorname{Res\,}}
\zeta_F(s) \!\left(\frac{\#S_n[2]}{n!}\right)^{r_1}
\!\!\left(\frac{1}{n!}\right)^{r_2}
\!\!\prod_{\fP\notin M_\infty}
\!\!\left(\sum_{k=0}^n\frac{q(k,n-k)-q(k-1,n-k+1)}{\mbox{\bf N} \,\fP^k}\right).
\end{equation}
\item[{\rm (b)}] If $F$ is a function field with field of constants $\F_q$, then
\begin{equation}\label{eq:functionfieldcount}
\lim_{m\rightarrow\infty}\frac{N_n(F,q^{2m})}{q^{2m}} = (\log q)\,
\underset{s=1}{\operatorname{Res\,}}
\zeta_F(s)\prod_{\fP} \left(\sum_{k=0}^n\frac{q(k,n-k)-q(k-1,n-k+1)}{\mbox{\bf N} \,\fP^k}\right).
\end{equation}
\end{itemize}
When $F$ is a function field, $X$ only runs through the possible norms of relative discriminants in the above limit.
\end{theoremx}
\noindent
Here $S_n[2]$ denotes the number of elements of order dividing 2 in
the symmetric group on $n$ letters.  We remark that the weight
$(\#\Aut(L/F))^{-1}$ is equal to~$1$ for $n\geq3$ and is $1/2$ for
$n=2$. The inclusion of these weights yields a statement of
Theorem~\ref{thmainfield} that is uniform in $n$. We note that Theorem \ref{thmainfield} is new even for $n=2$ and $n = 3$ when $F$ has characteristic dividing $6$. It is new for $n = 4$ and $n = 5$ for all $F\neq\Q$.

We note that Theorem~\ref{thmainfield}(a) proves
\cite[Conjecture~A]{Bmass} for $n\leq 5$.  We conjecture that both
parts of Theorem~\ref{thmainfield} hold for all $n$.

For a global field $K$, one may define its {\it absolute
  discriminant} $D_K$ as follows. If $K$ is a number field, then $D_K$
is the absolute value of the discriminant of $K$ viewed as an
extension of $\Q$. If $K$ is the function field of a smooth projective
and geometrically connected curve $\sc$ over $\F_q$, then
$D_K=q^{2g-2}$ where
$g$ is the genus of $\sc$.
The norm of the
relative discriminant of $L$ over $F$ is related to the absolute
discriminants of $L$ and $F$ by the following formula (\cite{JP}):
$$D_L/D_F^{[L:F]} = N(\Disc(L/F)).$$ Therefore, the version of
Theorem~\ref{thmainfield} where we instead order degree-$n$ extensions
of $F$ by absolute discriminant is a consequence of
Theorem~\ref{thmainfield}.

We actually prove much more general versions of
Theorem~\ref{thmainfield}.  First, we count field extensions of~$F$ of
degree $n$ satisfying any collection of local conditions at finitely
many places. In fact, we will even allow for certain collections of
local conditions at infinitely many places. For each place $\fP$ of $F$, let $\Sigma_\fP$ be a set of isomorphism
classes of \'etale algebras of degree $n$ over~$F_\fP.$ We say that
the collection $(\Sigma_\fP)$ is {\it acceptable} if, for all but
finitely many $\fP$, the set $\Sigma_\fP$ contains all \'etale
algebras of degree $n$ over $F_\fP$ that are unramified or have
splitting type $(1^2\tau)$ where $\tau$ is an unramified splitting
type of dimension $n-2$. Note when the characteristic of $F$ is not
$2$, one of these two splitting possibilities occurs if the discriminant is squarefree. We prove the following asymptotic formula for the number of isomorphism classes of extensions $L$ over $F$
of degree $n$ of bounded relative discriminant whose local
specifications lie in $(\Sigma_\fP)$, i.e., $L\otimes
F_\fP\in\Sigma_\fP$ for all $\fP$.

\begin{theoremx}\label{thmainfield2}
Let $F$ be a global number or function field of any characteristic. Let $n=2$, $3$, $4$, or~$5$.  Let $\Sigma=(\Sigma_\fP)$ be an acceptable
collection of local specifications for degree-$n$ extensions of
$F$. Let $N_{n,\Sigma}(F,X)$ denote the number of degree-$n$ field
extensions $L$ with local specifications in $\Sigma$, weighted by
$(\#\Aut(L/F))^{-1}$, whose normal closure over $F$ has full Galois
group $S_n$, such that the norm of the relative discriminant of
$L$ over $F$ is at most $X$ if $F$ is a number field and is equal
to~$X$ if $F$ is a function field. Then
\begin{equation}\label{eq:combinedcount}
\lim_{X\rightarrow\infty} \frac{N_{n,\Sigma}(F,X)}{X} = c\,
\underset{s=1}{\operatorname{Res\,}}
\zeta_{F}(s)\prod_\fP m_{\fP}(\Sigma_\fP),
\end{equation}
where the constant $c$ is $\frac12$ if $F$ is a number field and $\log
q$ if $F$ is a function field, and where
\[
m_{\fP}(\Sigma_\fP) = \begin{cases}
\displaystyle\frac{N\fP-1}{N\fP} \sum_{K\in\Sigma_\fP}  \frac{|\Disc(K/F_\fP)|_\fP}{\#\Aut(K/F_\fP)} & \text{if }\fP\notin  M_\infty,\\[.2in]
\displaystyle\sum_{K\in\Sigma_\fP}  \frac{1}{\#\Aut(K/F_\fP)} & \text{if }\fP\in M_\infty.
\end{cases}
\]
When $F$ is a function field, $X$ only runs through the possible norms
of relative discriminants in the above limit.
\end{theoremx}
\noindent
If, for every $\fP$, the set $\Sigma_\fP$ consists of all \'etale
extensions of degree $n$ of $F_\fP$, then the local masses
$m_\fP=m_{\fP}(\Sigma_\fP)$ have been computed in \cite{Bmass}, and
Theorem~\ref{thmainfield2} reduces to Theorem~\ref{thmainfield}. Thus
Theorem~\ref{thmainfield2} gives an interpretation of the constants
appearing in the asymptotics in Theorem~\ref{thmainfield} in terms of
local masses.  Theorem~\ref{thmainfield2} also yields the density of
discriminants of degree-$n$ extensions of $F$ having squarefree relative discriminant.

Theorem \ref{thmainfield2} can be used to prove a complement to the
Chebotarev density theorem. If $F$ is a global field, and $L$ is an
$S_n$-extension of $F$ unramified at a finite prime $\fP$ of $F$, then
the Artin symbol at $p$ for the Galois closure of $L$ over $F$ is
defined as a conjugacy class in $S_n$. The Chebotarev density theorem
asserts that for fixed $L$ and varying $\fP$, the value of the
corresponding Artin symbol is equidistributed amongst the conjugacy
classes of $S_n$, where each class is weighted by its size. We prove
the analogous result for fixed $\fP$ and varying $L$.
\begin{theoremx}
    Let $F$ be a global number or function field of any characteristic and let $\fP$ be a finite prime of $F$.
    For $n=2$, $3$, $4$, or $5$, let $\Sigma$ be an
    acceptable collection of local specifications for degree-$n$
    extensions of $F$ such that $\Sigma_\fP$ consists of all
    unramified degree-$n$ \'etale extensions of $F_\fP$. Then, as $L$
    varies over degree-$n$ field extensions with local specifications
    in $\Sigma$ whose Galois closures over $F$ have Galois group $S_n$,
    the corresponding Artin symbol at $\fP$ is equidistributed across
    the conjugacy classes of $S_n$ where each conjugacy class is
    weighted by its size.
\end{theoremx}

We prove a further generalization of Theorem~\ref{thmainfield2}.  Let
$\Sigma$ denote again any acceptable collection of local
specifications of degree-$n$ extensions of $F$, and let $S$ be any
nonempty finite set of places of~$F$ containing~$M_\infty$. Define the
{relative $S$-discriminant} $\Disc_S(L/F)$ of an extension $L$ of $F$
in the usual way. Then the norm of $\Disc_S(L/F)$ is given by
\begin{equation}\label{eqnormdiscS}
N(\Disc_S(L/F)):=\prod_{\fP\notin S}\prod_{w\mid \fP} |\Disc(L_w/F_\fP)|_\fP^{-1}.
\end{equation}
We prove:

\begin{theoremx}\label{thmainfieldS}
Let $F$ be a global number or function field of any characteristic. Let $n=2$, $3$, $4$, or $5$.  Let $\Sigma=(\Sigma_\fP)$ be an acceptable
collection of local specifications for degree-$n$ extensions of
$F$. Let $S$ be a nonempty finite set of places of $F$ containing
$M_\infty$. Let $N_{n,\Sigma,S}(F,X)$ denote the number of degree-$n$
field extensions $L$ with local specifications in $\Sigma$, weighted
by $(\#\Aut(L/F))^{-1}$, whose normal closure over $F$ has full Galois
group~$S_n$, such that the norm of the relative $S$-discriminant of
$L$ over $F$ is at most $X$ if $F$ is a number field, and is equal
to~$X$ if $F$ is a function field. Then
\begin{equation}\label{eq:combinedcount2}
\lim_{X\rightarrow\infty} \frac{N_{n,\Sigma,S}(F,X)}{X} = c\,
\underset{s=1}{\operatorname{Res\,}}
\zeta_{F,S}(s)\prod_\fP m_{\fP,S}(\Sigma_\fP),
\end{equation}
where $\zeta_{F,S}(s)=\prod_{\fP\notin S}(1 - (N\fP)^{-s})^{-1}$ is
the partial zeta function and the constant $c$ is $\frac12$ if $F$ is
a number field and $\log q$ if $F$ is a function field, and where
\begin{equation}\label{eq:localmassold}
m_{\fP,S}(\Sigma_\fP) = \begin{cases}
\displaystyle\frac{N\fP-1}{N\fP} \sum_{K\in\Sigma_\fP}  \frac{|\Disc(K/F_\fP)|_\fP}{\#\Aut(K/F_\fP)} & \text{if }\fP\notin  S,\\[.2in]
\displaystyle\sum_{K\in \Sigma_\fP}  \frac{1}{\#\Aut(K/F_\fP)} & \text{if }\fP\in S.
\end{cases}
\end{equation}
When $F$ is a function field, $X$ only runs through the possible norms
of relative $S$-discriminants in the above limit.
\end{theoremx}

We will use Theorem~\ref{thmainfieldS} to prove
Theorem~\ref{thmainfield2} (and hence Theorem~\ref{thmainfield}).
When $F$ is a number field and $S=M_\infty$, Theorem
\ref{thmainfieldS} reduces to Theorem~\ref{thmainfield2}. When $F$ is
the function field of a smooth projective and geometrically connected
curve $\sc$ over $\F_q$, we set $S$ to be any nonempty finite set of
closed points and use Theorem~\ref{thmainfieldS} to count the number
of degree-$n$ extensions~$L$, or equivalently, degree-$n$ covers
$\sc'$ of $\sc$, having a fixed genus and prescribed
splitting/ramification behavior at the chosen points; we then sum over
all the possible splitting/ramification behaviors at these points to
obtain Theorem~\ref{thmainfield2}.
  Theorem~\ref{thmainfieldS}, with $S$ taken to be the union of
  $M_\infty$ and all the places above $2$, also allows a new, clean
  proof of Theorems~\ref{thmainfield} and \ref{thmainfield2} via
  geometry of numbers when $F$ is a number field and $n=2$.

Let $\cO_S$ denote the ring of $S$-integers in $F$. In the case of
function fields, $\cO_S$ is the ring of regular functions on the
affine curve obtained by removing the closed points in $S$. For any
degree-$n$ extension $L$ of $F$, the integral closure $\cO_{L,S}$ of
$\cO_S$ in $L$ is a projective module over $\cO_S$ of rank $n$. The
structure theory of projective modules over a Dedekind domain
\cite[Theorem 1.6]{Mlnr} states that $\cO_{L,S}\simeq \cO_S^{n-1}\times
I$ as $\cO_S$-modules for some fractional ideal $I$ of $\cO_S$. The
ideal class of $I$ is an invariant of $L$, called the {\it $S$-Steinitz class} of $L$. As a further
byproduct of the proofs of Theorems \ref{thmainfield} and
\ref{thmainfieldS}, we obtain the following equidistribution result.

\begin{theoremx}\label{thm:equidistribution}
Let $F$ be a global number or function field of any characteristic. Let $S$ be a nonempty finite set places of
$F$ containing $M_\infty$. For any $n=2$, $3$, $4$ or $5$, let
$\Sigma$ be an acceptable collection of local
specifications for degree-$n$ extensions of $F$. Then the $S$-Steinitz
classes of degree-$n$ field extensions of $F$ with local
specifications in $\Sigma$, whose normal closure has full Galois group
$S_n$, are equidistributed in the class group of $\cO_S$.
\end{theoremx}
\noindent
The result above was proven in the cases $n=2$ or $3$ and $S=M_\infty$
for the full family of degree-$n$ field extensions of $F$, when the characteristic of $F$ does not divide $6$, by Kable and
Wright~\cite{KW}.

The cases $n=3$ and $n=4$ of Theorems \ref{thmainfield} and
\ref{thmainfield2} yield results on the average sizes of the
$3$-torsion subgroups of the relative class groups of quadratic
extensions of $F$ and the $2$-torsion subgroups of the relative class
groups of cubic extensions of $F$. 

For a finite extension $L/F$,
recall that the relative class group $\Cl(L/F)$ is the kernel of the
relative norm map $N_{L/F}:\Cl(L)\to\Cl(F)$ from the class group of
$L$ to the class group of $F$. Taking the kernel of the relative norm map between the narrow class groups of $L$ and $F$
yields the relative narrow class group $\Cl^+(L/F)$. For a prime $p$,
we let $h_p(L/F)$ and $h_p^+(L/F)$ denote the sizes of the $p$-torsion
subgroups of $\Cl(L/F)$ and $\Cl^+(L/F)$, respectively. 

To state
these theorems, we need to introduce some additional notation. Let $F$
be a global field, and let $\Sigma=(\Sigma_\fP)$ be an acceptable set of
local specifications for degree-$n$ extensions of $F$. We say that
$\Sigma$ is {\it archimedeally pure} if for every $\fP\in M_\infty$,
the set $\Sigma_\fP$ consists of a single element. For $n=2$ and
$n=3$, we define the quantities $\alpha_n(\Sigma)$ as follows:
\begin{equation}\label{eqalphapure}
  \alpha_2(\Sigma):=\#\{\fP:F_\fP=\R;\,\Sigma_\fP=\{\R^2\}\};\;\;\;
  \alpha_3(\Sigma):=\#\{\fP:F_\fP=\R;\,\Sigma_\fP=\{\R^3\}\}.
\end{equation}
If $F$ is a function field, then $\alpha_n(\Sigma)=0$. We prove the following theorem:
\begin{theoremx}\label{thclassgps}
Let $n=2$ or $3$, and let $F$ be a fixed global field of any characteristic. If $F$ is a
number field, we denote the number of real completions of $F$ by $r_1$
and the number of conjugate pairs of complex completions of $F$ by
$r_2$. If $F$ is a function field, we set $r_1=r_2=0$. Let
$\Sigma$ be an acceptable archimedially pure collection of
local specifications for degree-$n$ extensions of $F$. Let
$S_{n,\Sigma}(F,X)$ denote the set of degree-$n$ field extensions $L$
with local specifications in $\Sigma$, whose normal closure over $F$
has full Galois group $S_n$, such that the norm of the relative
discriminant of $L$ over $F$ is at most $X$ if $F$ is a number field,
and is equal to~$X$ if $F$ is a function field. Then
\begin{itemize}
  \item[{\rm (a)}] If $n=2$, then
\begin{equation*}
  \displaystyle\lim_{X\to\infty}\frac{\displaystyle\sum_{L\in S_{2,\Sigma}(F,X)}h_3(L/F)}
                   {\displaystyle\sum_{L\in S_{2,\Sigma}(F,X)}1}\;\;=\;\;1+3^{-r_2-\alpha_2(\Sigma)\phantom{-r_1}}.
\end{equation*}
  \item[{\rm (b)}] If $n=3$, then
\begin{equation*}
  \begin{array}{rcl}
  \displaystyle\lim_{X\to\infty}\frac{\displaystyle\sum_{L\in S_{3,\Sigma}(F,X)}h_2(L/F)}
                   {\displaystyle\sum_{L\in S_{3,\Sigma}(F,X)}1}&=&1+2^{-r_1-2r_2-\alpha_3(\Sigma)};\\[.5in]
  \displaystyle\lim_{X\to\infty}\frac{\displaystyle\sum_{L\in S_{3,\Sigma}(F,X)}h^+_2(L/F)}
                   {\displaystyle\sum_{L\in S_{3,\Sigma}(F,X)}1}&=&1+2^{-r_1-2r_2+\alpha_3(\Sigma)}.
  \end{array}
\end{equation*}
\end{itemize}
\end{theoremx}
When $F=\Q$, Part (a) of the above theorem was a result of
Davenport--Heilbronn \cite{DH} while Part (b) was proved in
\cite{dodqf}. For general global fields $F$ having characteristic not
2 or 3, Part (a), for the full family of quadratic extensions of $F$, is due to
Datskovsky--Wright \cite{DW}. The above result has the following immediate consequence.
\begin{corollaryx}
Let $F$ be a number field, and let notation be as in Theorem \ref{thclassgps}. Then as $L$ varies over $S_{3,\Sigma}(F,X)$, a positive proportion of the relative class numbers of $L$ over $F$ are odd.
\end{corollaryx}

Next, we explore the average number of certain unramified {\em
  nonabelian} extensions of quadratic extensions of a global field
$F$. More precisely, given finite groups $G\subset G'$ and a quadratic
extension $L$ of $F$, we say that $K$ is a $(G,G')$-extension of $L$
if $K$ is Galois over $L$ with Galois group $G$, and the Galois
closure of $K$ over $F$ has Galois group $G'$. The average number of
$(G,G')$-extensions of quadratic fields over $\Q$, for
$(G,G')=(A_n,S_n)$, $(G,G')=(S_n,S_n\times C_2)$ for $n=3,4,5$, were
determined in \cite{Bgeosieve}. Here, we prove the analogous results
for acceptable families of quadratic extensions of global fields $F$.
\begin{theoremx}\label{thurna}
  Let $F$ be a fixed global field. If $F$ is a number field, we denote
  the number of real completions of $F$ by $r_1$ and the number of
  conjugate pairs of complex completions of $F$ by $r_2$. If $F$ is a
  function field, we set $r_1=r_2=0$. Let $\Sigma$ be an
  acceptable archimedially pure collection of local specifications for
  quadratic extensions of $F$. For $n=3$, $4$, or $5$, let
  $E_\Sigma(G,G')$ denote the average number of unramified
  $(G,G')$-extensions of quadratic extensions $L$ of $F$ having local
  specifications in~$\Sigma$, where these fields $L$ are ordered by
  the norms of their relative discriminants over $F$. Then
  \begin{equation*}
    \begin{array}{rclcl}
      &{\rm (a)}& \displaystyle E_\Sigma(A_n,S_n)&=&\displaystyle\frac{1}{2}\cdot
      \displaystyle\Bigl(\frac{2}{n!}\Bigr)^{r_2+\alpha_2(\Sigma)}\cdot
      \displaystyle\Bigl(\frac{1}{(n-2)!}\Bigr)^{r_1-\alpha_2(\Sigma)};\\[.2in]
      &{\rm (b)}& E_\Sigma(S_n,S_n\times C_2)&=&\infty.
    \end{array}
  \end{equation*}
\end{theoremx}
Note that the average values in Theorems \ref{thclassgps} and
\ref{thurna} do not change with local conditions imposed at
finite places. The reason we have restricted to archimedially pure
local specifications is that the averages do depend on local
specifications at infinite places.

Finally, we consider the following question, settled in the case $n=3$ by work of Wood \cite{WoodFinite}. Let $n$ be a positive integer, let $C$ be a smooth curve over $\F_q$ with characteristic greater than $n$, and let $z\in C(\F_q)$ be fixed. In \cite[Conjecture 3.2]{WoodFinite}, Wood predicts the expected number of points in the fiber over $z$ of a random curve over $\F_q$ with a degree-$n$ map to $C$ with full monodromy, as the genus goes to infinity. Wood proves that this conjecture follows from \cite[Conjecture 5.1]{Bmass}. Since this latter conjecture for $n=2$, $3$, $4$ or $5$ and arbitrary characteristic follows from Theorem \ref{thmainfield2}, we have the following unconditional result.

\begin{theoremx}\label{thfinpts}
Let $n=2$, $3$, $4$ or $5$, and let $C$ be a smooth curve over $\F_q$. Fix a point $z\in C(\F_q)$. Then the expected number of points in the fiber over $z$ of a random curve over $\F_q$
with a degree $n$ map to $C$ and full monodromy $($as the genus goes to infinity$)$ is
\begin{equation*}
\begin{array}{ccc}
    1&\mbox{ if }&n=2,\\ [.075in]
    1+\frac{q}{q^2+q+1}&\mbox{ if }&n=3,\\ [.05in]
    1+\frac{q^2+q}{q^3+q^2+2q+1}&\mbox{ if }&n=4,\\ [.05in]
    1+\frac{q^3+2q^2+2q}{q^4+q^3+2q^2+2q+1}&\mbox{ if }&n=5.
\end{array}
\end{equation*}
\end{theoremx}
\medskip

\noindent\textbf{Methods.} Our proofs of Theorems \ref{thmainfield}--\ref{thfinpts} extend the geometry-of-numbers techniques
developed in \cite{DH,dodqf,dodpf} so that they can be applied over
arbitrary global fields. More precisely, we consider prehomogeneous
representations $V_n$ of split (over $\Z$) reductive groups $G_n$ for
$2\leq n\leq 5$ such that the $G_n(F)$-orbits of $V_n(F)$ of nonzero
discriminant parametrize degree-$n$ \'etale extensions of $F$, with the exception in the case $n=2$ and $F$ has characteristic $2$ where the group $G_2$ is non-reductive. These
representations first arose in a unified context in the work of
Sato--Kimura~\cite{SK}, and the connection with field extensions was
first studied systematically in the work of Wright--Yukie~\cite{WY}.
The integral orbits of these representations were classified in
\cite{DF,BII,BIII}, and the rational orbits were then counted using
this classification of integral orbits via suitable
geometry-of-numbers arguments in \cite{DH,dodqf,dodpf}.  In this
paper, it is these latter parametrization and counting methods that we
aim to extend to general base fields.

One key difference in working over an arbitrary global field $F$
rather than $\Q$ is that not all $\cO_S$-modules are free. In these
cases, it is no longer sufficient to consider just the $S$-integral
orbits $G_n(\cO_S)\backslash V_n(\cO_S)$ to parametrize degree-$n$
extensions of $\cO_S$.  The parametrization of ring extensions of
degrees up to $4$ over an arbitrary base is due to Wood
\cite{WoodQuartic,WoodThesis}, and this parametrization could
be used to prove our main theorems in these cases. However, we use a
different approach and instead consider, for each ideal class $\beta$
of $\cO_S$, orbits $\Gamma_\beta\backslash \L_\beta^{\rm max}$, where
$\Gamma_\beta$ is a subgroup of $G_n(F)$ commensurable with
$G_n(\cO_S)$ and $\L_\beta^{\rm max}$ is a $\Gamma_\beta$-invariant
subset of an additive subgroup $\L_\beta$ of $V_n(F)$ commensurable
with $V_n(\cO_S)$. The subset $\L_\beta^{\rm max}$ of $\L_\beta$ is
defined by congruence conditions at every prime $\fP$ of
$\cO_S$. These $\Gamma_\beta$-orbits in $\L_\beta^{\rm max}$ are then
shown to correspond to $S_n$-extensions whose $S$-Steinitz class is
$\beta$. We then sum over $\beta$. Our approach works uniformly for all degrees $n\leq 5$.

To count the number of $\Gamma_\beta$-orbits in $\L_\beta$, we view
$\L_\beta$ as a lattice in $V_n(F_{S})$ where $F_S=\prod_{\fP\in S}
F_\fP$. We construct a fundamental domain $R$ for the action of
$\Gamma_\beta$ on $V_n(F_{S})$ so that the set of $\L_\beta$-points
in this fundamental domain is in bijection with the set of
$\Gamma_\beta$-orbits in $\L_\beta$. We then wish to obtain
asymptotics for the number of these orbits of bounded size by
computing the volume of the subset of the fundamental domain having
bounded size; here by ``size'', we mean the norm of the relative $S$-discriminant of the corresponding degree-$n$ field extension. This is much easier to carry out when the relevant subset of the
fundamental domain is bounded, which is true only in the case
$n=2$. When $n=3,4,5$, the fundamental domains contain cusps that go
off to infinity and we show that the number of points in these cusps
corresponding to $S_n$-extensions of $F$ is negligible. When $n=3$, it
turns out that no points in the cusp region correspond to
$S_n$-extensions. When $n=4$, there are two cuspidal regions that
contain points corresponding to $S_n$-extensions, and when $n=5$,
there are hundreds of such cusps! We reduce this cusp analysis to
the verification of certain combinatorial conditions on the characters
of the maximal torus of the derived subgroup of $G_n$ that we prove
hold over a general global field if and only if they hold over $\Q$
(using the fact that $G_n$ is split over $\Z$).  Since these
conditions have been checked over $\Q$ in \cite{DH,dodqf,dodpf}, it
follows that these conditions hold over arbitrary global fields.

To obtain the number of $\Gamma_\beta$-orbits in $\L_\beta^{\rm max}$
from the number of $\Gamma_\beta$-orbits in $\L_\beta$, we impose
congruence conditions at every prime ideal $\fP$ of $\cO_S$ by
multiplying together density functions at each $\fP$. Since an
infinite number of conditions are imposed, a uniformity estimate on
the error term is required. For $n=2$, this is elementary. For
$n=3,4,5$, such a uniformity estimate over $\Q$ has been established
in~\cite{Bgeosieve} and, as we will see, only some mild modifications
are required to generalize this estimate to arbitrary global fields.

To finish the proof, we must compute the volumes of some subsets of $V_n(F_{S})$
and $V_n(F_\fP)$ for $\fP\notin S$. We use a Jacobian change-of-variable formula to transfer these volume computations to $G_n$. Once
we multiply all the local volumes together and sum over all the
$S$-Steinitz classes, we find that we obtain essentially the Tamagawa
number of $G_n$, which in all our cases turns out to be $1$.  This
completes the proof of Theorem~\ref{thmainfieldS}, and thus also of
Theorems~\ref{thmainfield} and \ref{thmainfield2}.

\medskip 

This article is organized as follows. In Section \ref{sec:prelim}, we
set up some notations, recall the formula for the local masses
$m_\fP$, and define the Tamagawa number of a reductive
group over a global field. In Section \ref{sec:param}, we give the
representations $(G_n,V_n)$ for $n=2,3,4,5$ and describe the
parametrization of $S_n$-extensions via rational orbits and
$S$-integral orbits. In Section~\ref{sec:count}, we count the number
of $S$-integral orbits and carry out the cusp analysis as described
above. In Section~\ref{sec:sieve}, we impose the necessary congruence
conditions and prove the corresponding uniformity estimates needed to perform our sieve. In Section~\ref{sec:fieldscounting}, we specialize to the
congruence conditions required for counting field extensions.  In
Section~\ref{sec:volume}, we carry out the corresponding local volume
computations in terms of the Tamagawa number of the group and the
local masses. In Section~\ref{sec:proof}, we combine together the
results of the previous sections to prove our main theorems with the
restriction that the characteristic of $F$ is not $2$ if $n=2$. In
Section~\ref{sec:quadratic}, we introduce a new representation that
parametrizes quadratic extensions that works in any characteristic and
prove our main theorems in the case $n=2$ with no restriction on the
characteristic of $F$.

We note that our methods may be applied to many other counting
problems as well.  For example, they may be used to study the
3-torsion in the $S$-class groups of orders in quadratic extensions of any base field, as well as the 2-torsion in the $S$-class groups of orders in cubic extensions, as studied in \cite{DH,dodqf} when $F = \Q$. Finally, the counting methods
can eventually be extended to representations over $F$ having more than one invariant, leading, e.g., to
a proof of the boundedness of the average rank of elliptic curves over an arbitrary global field $F$ and
related results for average Selmer group sizes and curves of higher genus. These directions and the
corresponding requisite extensions of the methods will be discussed in the second article of this series.

\section{Preliminaries}\label{sec:prelim}

\subsection{Notations}\label{sec:notation}
Throughout this paper, $F$ denotes a fixed global field. That is, $F$ is either a number field or the field of
rational functions of a smooth projective and geometrically connected
algebraic curve $\sc$ over $\F_q$.

For every place $\fP$ of $F$, let $F_\fP$ denote the completion of $F$
at $\fP$. When $\fP$ is nonarchimedean, let $\cO_\fP$ and
$k(\fP)$ denote the ring of integers and the residue field at $\fP$,
respectively. We define $\deg\fP = [k(\fP):\F_q]$ if $F$ is a function field. Let $N\fP=|k(\fP)|$ denote the norm of $\fP$. For any
$a\in F_\fP$, define its $\fP$-adic norm by $|a|_\fP =
(N\fP)^{-\fP(a)}$ where $\fP(a)$ denotes the $\fP$-adic valuation of
$a$. Since every fractional ideal $I$ of $\cO_\fP$ is generated by an
element $a\in F_\fP$ unique up to $\cO_\fP^\times$, we define
$|I|_\fP$ to be $|a|_\fP$. Note that this is the inverse of the ideal
norm. When $F$ is a number field and $\fP$ is archimedean, we define $\deg\fP = [F_\fP:\R]$ and define
$|a|_\fP=|N_{F_\fP/\R}(a)|_\infty$ for any $a\in F_\fP$ where $|\cdot|_\infty$ is the usual absolute value on $\R$. Then for any
$a\in F^\times$, we have the product formula $\prod_\fP |a|_\fP = 1.$

Let $\A$ denote the ring of adeles of $F$. For any adele $a=(a_\fP)$,
we write $|a|$ for the product $\prod_\fP |a_\fP|_\fP$. For any finite
set $S$ of places containing $M_\infty$, let $\A_S$ denote the ring of
$S$-adeles of $F$, that is, $\A_S$ is the restricted direct product of
$F_\fP$ for $\fP\notin S$. Denote by $F_{S}$ the product
$\prod_{\fP\in S}F_\fP$. We embed $\A_S$ and $F_{S}$ in $\A$ by
setting all other coordinates to $1$, and continue to write $|\cdot|$
for the restrictions of $|\cdot|$ on $\A$. On $F_{S}$, $|\cdot|$ is
also called the {\it $S$-norm}. For any $S$-adele $a=(a_\fP)$, we
write $(a)$ for the fractional ideal of $\cO_S$ such that $(a)\otimes
\cO_\fP = a_\fP\cO_\fP$ for every $\fP\notin S$. For any fractional
ideal $I$ of $\cO_S$, let $|I|$ denote the product of
$|I\otimes\cO_\fP|_\fP$ as $\fP$ runs through places not in $S$. Then
$|(a)| = |a|$ for any $a\in \A_S$.

\subsection{The local masses $m_\fP$}

For any $\fP\notin S$, the local masses $m_{\fP}(\Sigma_\fP)$, in the
case that $\Sigma$ is the set of all degree-$n$ separable extensions
of $F$ that are totally ramified at $\fP$, was computed in
\cite[Theorem~1]{Smass}.
When $F_\fP=\Q_p$ and $\Sigma$ consists of all degree-$n$ extensions
with no local restrictions, the corresponding full mass
$m_\fP=m_{\fP}(\Sigma_\fP)$ was computed in \cite[Theorem~1.1]{Bmass}.
The methods apply to general nonarchimedean local fields $F_\fP$ to
give
\begin{equation}\label{eq:localmassfinite}
\frac{N\fP}{N\fP - 1}m_\fP = \sum_{\substack{[K:F_\fP]=n\\K\textrm{\'etale}}}  \frac{|\Disc(K/F_\fP)|_\fP}{\#\Aut(K/F_\fP)} = \sum_{k=0}^{n-1}\frac{q(k,n-k)}{(N\fP)^k}.
\end{equation}
When $F_\fP=\R$, we have (\cite[Proposition~2.4]{Bmass})
\begin{equation}\label{eq:localmassR}
m_\fP = \sum_{\substack{[K:\R]=n\\K\textrm{\'etale}}}  \frac{1}{\#\Aut(K/\R)} = \frac{\#S_n[2]}{n!}.
\end{equation}
When $F_\fP=\C$, a degree-$n$ \'etale extension of $\C$ can only be $\C^n$ and so we have
\begin{equation}\label{eq:localmassC}
m_\fP = \sum_{\substack{[K:\C]=n\\K\textrm{\'etale}}}  \frac{1}{\#\Aut(K/\C)} = \frac{1}{n!}.
\end{equation}


\subsection{Tamagawa number of a reductive group over $F$}\label{sec:tam}

In this section, we first recall the definition of the Tamagawa number
of a reductive group $G$ over a global field $F$ where the (global)
character group has rank $1$. See \cite[Chapitre I, \S5]{Os} for the definition in more
general cases.

Let $\omega_G$ denote a top degree left-invariant differential form on
$G$ defined over $F$. For every place $\fP$, write $\omega_{G,\fP}$
for the associated Haar measure on $G(F_\fP)$. Let $S$ be a nonempty
finite set containing all the archimedean places. Then the Tamagawa
measure $\tau_G$ of $G(\A)$ is defined by
$$\tau_G = \frac{D_F^{-\frac{\dim
      G}{2}}}{\underset{s=1}{\operatorname{Res\,}}\zeta_S(s)}
\prod_{\fP\notin S}\frac{N\fP}{N\fP-1}\omega_{G,\fP} \prod_{\fP\in
  S}\omega_{G,\fP}.$$ We note that this measure is independent of the
choice of $S$. Let $\chi$ denote a basis element for the character
group. Let $G(\A)^1$ denote the subgroup of $G(\A)$ consisting of
elements $g$ such that $|\chi(g)|=1$. The Tamagawa number of $G$ is
defined to be the volume of $G(F)\backslash G(\A)^1$ with respect to a
Haar measure $\tau_G^1$ on $G(\A)^1$ defined as follows.

When $F$ is a number field, there is an embedding of $\R^+$ in $G(\A)$
sending each $\lambda\in \R^+$ to the adele $\iota(\lambda)$ so that
$|\chi(\iota(\lambda))| = \lambda$ and that via this embedding, one may
write $G(\A)$ as a Cartesian product $G(\A) \simeq G(\A)^1\times
\R^+.$ Then $\tau_G^1$ is the Haar measure of $G(\A)^1$ such
that $$\tau_G = \tau_G^1\times d^\times \lambda$$ where $d^\times
\lambda = d\lambda/\lambda$ is the usual multiplicative measure on
$\R^+$.

When $F$ is a function field with field of constants $\F_q$, $G(\A)^1$
is an open subgroup of $G(\A)$ and $\tau_G^1$ is the Haar measure on
$G(\A)^1$ defined by
$$\tau_G^1 = \frac{{\tau_G}_{|G(\A)^1}}{\log q}.$$

It is known that the Tamagawa number of a reductive group over $F$
whose derived group is a product of $\SL_k$'s is $1$
(\cite[Chapitre III, Th\'{e}or\`{e}me~5.3]{Os}).

\section{Parametrization of quadratic, cubic, quartic, and quintic extensions of a fixed global field}\label{sec:param}

Let $F$ be a fixed global field.  Our goal in this section is to
parametrize quadratic, cubic, quartic, and quintic extensions of $F$
in terms of certain orbits of prehomogeneous representations.
The representation we give here for $n=2$ does not parametrize quadratic extensions in the case when $F$ has characteristic $2$; however, in Section \ref{sec:quadratic}, we will describe a different (though slightly more complicated and non-reductive) representation that will apply in all characteristics.  We have chosen to keep the representation listed here for simplicity of exposition in the cases when the characteristic of $F$ is not 2. For $n=2$, $3$, $4$, and $5$, we define the prehomogeneous
representations $(G_n,V_n)$ as in Table \ref{tablereps}.

\begin{table}[hbt]
\centering
\begin{tabular}{|c | c|c|c|}
\hline
$n$ & $G_n$ & $V_n$ & Action\\
\hline
$2$ & $\bG_m$ & $\bG_a$ & $g\cdot v:=g^2v$\\[.05in]
$3$ & $\GL_2$ & $\Sym^3(2)$ & $g\cdot f(x,y):=\det(g)^{-1}f((x,y)\cdot g)$ \\[.05in]
$4$ & $\GL_2\times\GL_3/\{(\lambda^2,\lambda^{-1})\}$ & $2\otimes\Sym^2(3)$ &
$(g_2,g_3)\cdot (A,B):=(g_3Ag_3^t,g_3Bg_3^t)g_2^t$ \\[.05in]
$5$ & $\GL_4\times\GL_5/\{(\lambda^2,\lambda^{-1})\}$ & $4\otimes\wedge^2(5)$ &
$(g_4,g_5)\cdot (A,B,C,D):=(g_5Ag_5^t,g_5Bg_5^t,g_5Cg_5^t,g_5Dg_5^t)g_4^t$ \\
\hline
\end{tabular}
\caption{Prehomogeneous representations parametrizing degree-$n$ extensions}\label{tablereps}
\end{table}
\noindent
In Table \ref{tablereps}, $\lambda\in\GL_k$ denotes the
element in the center of $\GL_k$ with $\lambda$'s on the diagonal. In
the cases $n=4$ and $5$, the actions listed are the actions of
$\GL_2\times\GL_3$ and $\GL_4\times\GL_5$, respectively, on $V_n$. It
is easy to check that the subgroup $\{(\lambda^2,\lambda^{-1})\}$ of
$G_n$ acts trivially, and thus the action descends to an action of
$G_n$ on~$V_n$. Each of these groups $G_n$ are reductive, and their
(global) character groups $\Hom(G_n,\bG_m)$ all have rank $1$. Let $\chi$ be
a generator of $\Hom(G_n,\bG_m)$.  Then $\chi$ is unique up to
inversion and for $n=2$, $3$, $4$, and $5$, we may take $\chi$ to
respectively be
\begin{equation}\label{eqchi}
\begin{array}{rcl}
\chi(g)&=&g,\\[.05in]
\chi(g)&=&\det(g),\\[.05in]
\chi(g_2,g_3)&=&\det(g_2)^3\det(g_3)^4,\\[.05in]
\chi(g_4,g_5)&=&\det(g_4)^5\det(g_5)^8.
\end{array}
\end{equation}

The representations $(G_n,V_n)$ are {\it prehomogeneous}, that is, the
action of $G_n(\C)$ on $V_n(\C)$ consists of a single open orbit (in
the Zariski topology); see~\cite{SK}. It follows that for each $n$,
the ring of relative invariants for the representation $V_n$ of $G_n$
is generated by a single element which we refer to as the {\it
  discriminant} and denote by $\Delta=\Delta_n$. The condition of
being a relative invariant means that $\Delta_n(g\cdot
v)=\chi(g)^k\Delta_n(v)$ for $g\in G_n(\C)$ and $v\in V_n(\C)$ for $k$
independent of $g$ and $v$.  It is easy to check that $k=2$ in all our
cases. We note that the constant $c$ in Theorems \ref{thmainfield}, \ref{thmainfield2} and
\ref{thmainfieldS} when $F$ is a number field arises as $1/k$.

For any $v\in \bG_a$, we have $\Delta_2(v) = v$. For any binary cubic
form $ax^3+bx^2y+cxy^2+dy^3\in\Sym^3(2)$, we have
$$\Delta_3(ax^3+bx^2y+cxy^2+dy^3) = b^2c^2-4ac^3-4b^3d-27a^2d^2+18abcd.$$
For any pair ternary quadratic forms viewed formally as $3\times 3$ symmetric matrices
$$(A,B)=\left(\begin{pmatrix} a_{11} &\frac12 a_{12} &\frac12
  a_{13}\\\frac12 a_{12} & a_{22} &\frac12 a_{23}\\ \frac12 a_{13}
  &\frac12 a_{23} &a_{33}\end{pmatrix}, \begin{pmatrix} b_{11}
  &\frac12 b_{12} &\frac12 b_{13}\\\frac12 b_{12} & b_{22} &\frac12
  b_{23}\\ \frac12 b_{13} &\frac12 b_{23}
  &b_{33}\end{pmatrix}\right),$$ we have
$$\Delta_4(A,B) = \Delta_3(4\det(Ax - By)).$$ The polynomial
$\Delta_5$ is homogeneous in the coordinates of $5\times 5$
alternating matrices $A,B,C,D$ of total degree $40$.

The reason why the representations in Table \ref{tablereps} are
particularly interesting is due to the following parametrization
theorem.

\begin{theorem}\label{thfieldparamR}
Suppose $R$ is a ring. Then there is a map $\theta_R$ from the set of
$G_n(R)$-orbits on $V_n(R)$ to the set of rings of rank $n$ over
$R$. Suppose that $2$ is invertible in $R$ when $n=2$. Then $\theta$ has
the following properties:
\begin{itemize}
\item[{\rm (a)}] When $R$ is a PID, $\theta_R$ is a surjection and is a bijection
  when restricted to the preimages of maximal rank $n$ rings over
  $R$. Furthermore if $v\in V_n(R)$ with $\theta_R(v)$ maximal, then
\begin{equation*}
\Aut(\theta_R(v)/R)\cong\Stab_{G_n(R)}(v),
\end{equation*}
where $\Aut(\theta_R(v)/R)$ denotes the group of automorphisms of
$\theta_R(v)$ fixing $R$, and where for simplicity we write $\theta_R(v)$ for
$\theta_R(G_n(R)v)$.
\item[{\rm (b)}]  When $R$ is a Dedekind domain, we have for any $v\in V_n(R)$, $$\Disc(\theta_R(v)/R) = (\Delta_n(v)).$$
\end{itemize}
\end{theorem}
We say a rank $n$ ring over $R$ is maximal if it is
not strictly contained in any other rank $n$ ring over $R$.  In
particular, since every \'etale ring of rank $n$ over a field is
maximal, we have the following particular case of Theorem
\ref{thfieldparamR} for the parametrization of degree-$n$ fields,
which was the main theorem of Wright--Yukie~\cite{WY}:
\begin{theorem}\label{cor:fieldparam}
Let $K$ be a field. There is a bijection $\theta_K$ between the set of
$G_n(K)$-orbits on $V_n(K)$ having nonzero discriminant and the set of
\'etale degree-$n$ extensions of $K$. For any $v\in V_n(K)$, we
have $$\Aut(\theta_K(v)/K)\cong \Stab_{G_n(K)}(v).$$
\end{theorem}

\noindent \textit{Proof of Theorem~\ref{thfieldparamR}.} The set of
$G_n(\Z)$-orbits of a vector $v\in V_n(\Z)$ have been determined in
\cite{DF,BII,BIII,BIV} by explicitly writing down multiplication
tables for a corresponding ring of rank $n$ over $\Z$ and for its
resolvent ring. The same multiplication tables define in general rings
of rank $n$ over $R$ and its resolvent ring, for any $v\in V_n(R)$,
and agree with the construction of Wright--Yukie~\cite{WY} in Theorem
\ref{cor:fieldparam} when $R$ is a field. Forgetting about the
resolvent ring gives our map $\theta_R$. When $R$ is a PID, maximal
rank $n$ rings over $R$ have unique resolvent rings (\cite[\S15]{DF},
\cite[Corollary~5]{BIII}, and \cite[Corollary~3]{BIV}).  Hence
$\theta_R$, when restricted to the preimages of maximal rings, is a
bijection.

The stabilizer statement follows also from the construction of the
multiplication table. The case $n=2$ is immediate. We now consider the
case $n=3$, $4$, or $5$. Suppose $R$ is a PID. Let $v$ be an element
of $V_n(R)$ with $R'=\theta_R(v)$ maximal and hence has a unique
resolvent ring $S'$. Since $R$ is a PID, the quotients $R'/R$ and $S'/R$
are free. An automorphism of $R'$ over $R$ induces an automorphism of
$R'/R$ and an automorphism of $S'/R$ which under the identification of
these quotients as free $R$-modules gives an element of $G_n(R)$
stabilizing $v$. Conversely, an element $g$ of $G_n(R)$ induces an
automorphism of $R'$, as a free $R$-module, which is a ring
homomorphism if $g$ stabilizes $v$.

For the discriminant statement when $n>2$, from the multiplication
table we see that for any Dedekind domain $R$ and any $v\in V_n(R)$,
$\Disc(\theta_R(v)/R)$ is the ideal of $R$ generated by a polynomial
$\w{\Delta}_n(v)$ in the coordinates of $v$. Moreover, for any $g\in
G_n(R)$, one has $\w{\Delta}_n(g.v) = \chi(g)^2\w{\Delta}_n(v)$. Hence
$\w{\Delta}_n$ is a constant multiple of $\Delta_n$. Computing both at
a particular $v$ shows that this constant is a unit. The discriminant
statement when $n=2$ is clear.\hfill$\Box$

\vspace{5pt} We would like to describe a degree-$n$ extension $L$ of
$F$ integrally. By Corollary~\ref{cor:fieldparam}, there exists an
element $v\in V_n(F)$ unique up to the action of $G_n(F)$ such that
$\theta_F(v)=L$. Let $S$ be a fixed nonempty finite set of places
containing $M_\infty$. Suppose further that $S$ contains all the
places above $2$ if $n=2$. For any place $\fP\notin S$, let
$\cO_{L,\fP}$ denote the integral closure of $\cO_\fP$ in $L$ and let
$v_\fP$ be an element of $V_n(\cO_\fP)$ unique up to the action of
$G_n(\cO_\fP)$ such that $\theta_{\cO_\fP}(v_\fP) = \cO_{L,\fP}$.

We would like to patch these $v_\fP$ together into a global integral
element. Since $\theta_{F_\fP}(v) = \theta_{F_\fP}(v_\fP)=L_\fP$,
there exists $g_\fP\in G_n(F_\fP)$ such that $v_\fP = g_\fP v.$ The
adele $(g_\fP)\in G_n(\A_S)$ is well-defined up to translation by
$\prod_{\fP\notin S}G_n(\cO_\fP)$ on the left and by $G_n(F)$ on the
right and so it defines an element in the corresponding double coset
space $\cl_S$. We fix a representative $\beta\in G_n(\A_S)$ for each
double coset and write $\cl_S$ again for this set of
representatives. Then we have the decomposition
\begin{equation}\label{eqstrongapproximation0}
G_n(\A_S)=\coprod_{\beta\in\cl_S} \bigl(\prod_{\fP\notin S}G_n(\cO_\fP)\bigr)\beta G_n(F).
\end{equation}
The set $\cl_S$ is finite due to the works of Borel~\cite{Brl} and
Borel--Prasad~\cite{BP}. There then exists a unique $\beta\in\cl_S$
and some $g\in G_n(F)$ such that $(g_\fP) \in \prod_{\fP\notin
  S}G_n(\cO_\fP)\beta g$. The element $v'=g.v$ corresponds also to the
field $L$ and lies in
\begin{equation}\label{eq:Lbeta}
\L_\beta:= \displaystyle V_n(F)\cap \beta^{-1}\prod_{\fP\notin S} V_n(\cO_\fP)\prod_{\fP\in S}V_n(F_\fP).
\end{equation}
Moreover, as fractional ideals of $\cO_S$, we have
$$\Disc_S(L) = \prod_{\fP\notin S} \Disc(L_\fP/F_\fP) =
\prod_{\fP\notin S} (\Delta_n(v_\fP)) = (\chi(\beta)^2\Delta_n(v')).$$
Note that $\L_\beta$ is an additive subgroup of $V_n(F)$ commensurable
with $V_n(\cO_S)$.

We have now associated to a degree-$n$ extension $L$ over $F$ an
element $v'\in V_n(F)$ that is $\cO_S$-integral up to finite
denominators and such that $\Disc_S(L)$ and $\Delta_n(v')$ are equal
up to a fixed bounded factor. How unique is such a $v'$?

For nonarchimedean places $\fP$, we say that an element $v\in
V_n(\cO_\fP)$ is {\it maximal at $\fP$} if $\theta_{\cO_\fP}(v)$ is
maximal. We say that an element $v\in \L_\beta$ is {\it $S$-maximal}
if $\beta\cdot v$, when viewed as an element of $V_n(\cO_\fP)$ is
maximal at $\fP$ for all $\fP\notin S$. Denote the set of $S$-maximal
elements of $\L_\beta$ by $\L_\beta^\max$. Then it is clear that
$v'\in \L_\beta^\max.$

\begin{proposition}\label{prop:unique}
If $v_1,v_2\in V_n(\cO_\fP)$ are maximal, then any $g\in G_n(F_\fP)$
such that $g\cdot v_1=v_2$ is an element of $G_n(\cO_\fP)$. In
particular, If $v\in V_n(\cO_\fP)$ is maximal, then
$\Stab_{G_n(\cO_\fP)}(v)=\Stab_{G_n(F_\fP)}(v)$.
\end{proposition}

\begin{proof}
Since $v_1,v_2$ are maximal and correspond to the same field, they
also correspond to the same ring. By Theorem~\ref{thfieldparamR},
$\theta_{\cO_\fP}$ is a bijection when parameterizating maximal
rings. Hence there exists $g_0\in G_n(\cO_\fP)$ such that $g_0\cdot
v_1 = v_2.$ Then $g_0^{-1}g$ is in $\Stab_{G_n(F_\fP)}(v_1)$ and we
see that the two assertions in the proposition are in fact
equivalent. We prove the stabilizer statement instead. By
Theorem~\ref{thfieldparamR}, we have
$$\Stab_{G_n(\cO_\fP)}(v) \cong \Aut(\theta_{\cO_\fP}(v)/\cO_\fP) =
\Aut(\theta_{F_\fP}(v)/F_\fP) \cong \Stab_{G_n(F_\fP)}(v),$$as desired.\end{proof}

As a consequence, if $v_1$ and $v_2$ are two elements in
$\L_\beta^\max$ that are equivalent via some element $g$ of $G_n(F)$, then $g$ must be an
element of $\Gamma_\beta$, where
\begin{equation}
\Gamma_\beta:= \displaystyle G_n(F)\cap \beta^{-1}\bigl(\prod_{\fP\notin S}
G_n(\cO_\fP)\prod_{\fP\in S}G_n(F_\fP)\bigr)\beta.
\end{equation}
Therefore, we conclude that $v'$ is unique up to the action of
$\Gamma_\beta$.

We summarize the above discussion in the following theorem.

\begin{theorem}\label{thmainparam}
We have a bijection
\begin{equation*}
\{\mbox{degree-$n$ extensions of $F$ having nonzero
  discriminant}\}\longleftrightarrow \bigcup_{\beta\in\cl_S}
(\Gamma_{\beta}\backslash \L_{\beta}^\max).
\end{equation*}
Furthermore, if $L$ is a degree-$n$ extension corresponding to $v'\in
\L_\beta$ for some $\beta\in\cl_S$ under this bijection, then
$$\Disc_S(L)=(\chi(\beta))^2(\Delta_n(v')).$$
As a consequence, the $S$-Steinitz class of $L$ is the $\cO_S$-ideal
class of $(\chi(\beta))$.
\end{theorem}
\begin{proof}
Only the last statement on the $S$-Steinitz class needs to be
justified. When $n=2$, this is clear. When $n>2$, it follows
immediately from a result of Artin \cite{Artn} (see also
\cite[Proposition 3.3]{KW}): let $\delta_{L/F}\in F^\times$ be the
discriminant of the trace form with respect to some $F$-basis of $L$;
then there is a fractional ideal ${\mathfrak a}$ of $\cO_S$ such that
$\Disc_S(L)=(\delta_{L/F}){\mathfrak a}^2$ and moreover, the ideal
class of ${\mathfrak a}$ is the $S$-Steinitz class of $L$.
\end{proof}

\section{Counting $S$-integral orbits over global fields}\label{sec:count}

Let $F$ be a fixed global field.

We say that an element $v\in V_n(F)$ is {\it generic} if the
degree-$n$ extension $L$ of $F$ corresponding to $v$ is a field
extension of $F$, and the Galois group of the normal closure of $L$
over $F$ is $S_n$. For a subset $U$ of $V_n(F)$, we denote the set of
generic elements in $U$ by $U^\gen$ .

For the rest of this section, we fix a nonempty finite set $S$ of
places of $F$ containing $M_\infty$. 
Since the purpose of this section is to count the number of
$S$-integral orbits instead of counting field extensions, we make no
assumption on the characteristic of $F$ or $S$ when $n=2$. For any
positive real number $X$ and any subset $U$ of $V_n(F)$, we denote by
$U_X$ the set of elements of $U$ whose discriminants have $S$-norm
less than $X$ when $F$ is a number field and $S$-norm equal to~$X$
when $F$ is a function field.  If $H$ is a subgroup of $G_n(F)$ that
preserves $U$, then we let $N(U,H;X)$ denote the number of generic
$H$-orbits on $U_X$, where each orbit $H\cdot v$ is counted with
weight $1/\#\Stab_H(v)$.

By Theorem~\ref{thmainparam}, it suffices to count the number of
$S$-integral orbits such that the $S$-norm of the $\Delta_n$-invariant
is bounded. Since we need to impose congruence conditions later, we
consider the more general case of a sublattice $\L$ of $V_n(F)$ that
is commensurable with $V_n(\cO_S)$ and a subgroup $\Gamma$ of $G_n(F)$
preserving $\L$ that is commensurable with $G_n(\cO_S)$. Our goal in
this section is to determine asymptotics for $N(\L,\Gamma,X)$. We do
this by constructing a fundamental domain for the action of $\Gamma$
on $V_n(F_S)$ and counting the number of $\L$-points via suitable
geometry-of-numbers arguments.

\subsection{Reduction Theory}\label{sec:fund}

Let $S$ denote a set of places of $F$ containing $M_\infty$, and
let $F_S=\prod_{\fP\in S} F_\fP$.  In this section, we construct a
fundamental domain for the action of $\Gamma$ on $V_n(F_S)$ by
constructing a fundamental domain $\FF$ for the action of $\Gamma$ on
$G_n(F_S)$ via left multiplication, a fundamental domain $R$ for the
action of $G_n(F_S)$ on $V_n(F_S)$, and then taking $\FF\cdot R$ as a
multiset (so it is in bijection with $\FF\times R$).

The set $R$ can be taken as a finite discrete set. By
Theorem~\ref{thfieldparamR}, the set of $G_n(F_{S})$-orbits on
$V_n(F_S)$ is in bijection with the set of all
$S$-\textit{specifications}. Here an $S$-specification is a collection
$(L_\fP)_{\fP\in S}$ of degree-$n$ \'etale extensions $L_\fP$ of
$F_\fP$ for each $\fP\in S$. For each such $S$-specification $\sigma$,
we let $V_n(F_{S})^{(\sigma)}$ denote the set of elements
$(v_\fP)_{\fP\in S}\in V_n(F_S)$ such that the extension corresponding
to $v_\fP$ is $L_\fP$ for every $\fP\in S$. We also fix some
$v_\sigma$ in each $V_n(F_{S})^{(\sigma)}$. Then we may take $R$ to be
the collection of the $v_\sigma$'s. In what follows, we will fix an
$S$-specification $\sigma$ and consider only $V_n(F_S)^{(\sigma)}.$ We
define $\Aut(\sigma)$ by
\begin{equation*}
\Aut(\sigma):=\prod_{\fP\in S}\Aut(L_\fP).
\end{equation*}
Then the stabilizer in $G_n(F_S)$ of every element in $V_n(F_S)^{(\sigma)}$ is
isomorphic to $\Aut(\sigma)$ by Theorem~\ref{thfieldparamR}. We have
the following theorem whose proof is identical to that of
\cite[\S2.1]{BS2}.
\begin{theorem}\label{thmfunddomain}
Let $\FF$ denote a fundamental domain for the action of $\Gamma$ on
$G_n(F_{S})$.  For any fixed $v_\sigma\in V_n(F_{S})^{(\sigma)}$, the
multiset $\FF\cdot v_\sigma$ is an $\#\Aut(\sigma)$-fold cover of a
fundamental domain for the action of $\Gamma$ on
$V_n(F_{S})^{(\sigma)}$. More precisely, for $v\in V_n(F_{S})$, we
have
$$\#\{g\in\FF:g\cdot v_\sigma=v\}=\#\Aut(\sigma)/\#\Stab_\Gamma(v).$$
\end{theorem}

Fix any positive real number $X$. Recall the basis $\chi$ for the
group of characters of $G_n$ defined in \eqref{eqchi} and define
$G_n(\A)^1$ to be the subset of $G_n(\A)$ consisting of adeles $g$
such that adele norm $|\chi(g)|$ is $1$. Similarly define $G_n(F_S)^1$
via its natural embedding in $G_n(\A)$. Let $G_n(F_S)^{(\sigma)}_{X}$
denote the subset of $G_n(F_S)$ consisting of $S$-adeles $g$ such that
the adele norm $|\chi(g)|^2$ is at most (resp., equal to)
$X/|\Delta_n(v_\sigma)|$ when $F$ is a number field (resp., when $F$
is a function field). For any $g\in G_n(\cO_S)$, we have
$|\chi(g)|_\fP=1$ for every $\fP\notin S$ while the product
$\prod_\fP|\chi(g)|_\fP=1$ by the product formula. Hence the group
homomorphism $\chi_S:G_n(F_S)\rightarrow \R^+$ defined by $\chi_S(g) =
\prod_{\fP\in S}|\chi(g)|_\fP$ is trivial on $G_n(\cO_S)$. Since
$\Gamma$ is commensurable to $G_n(\cO_S)$, we see that $\chi_S(g)$ is
a root of unity for any $g\in\Gamma$. Hence $\Gamma$ is contained in
the kernel of $\chi_S$ and preserves $G_n(F_S)^{(\sigma)}_{X}.$ Let
$\FF(X)$ denote a fundamental domain for the action of $\Gamma$ on
$G_n(F_S)^{(\sigma)}_{X}$, then $\FF(X)\cdot v_\sigma$ is an
$\#\Aut(\sigma)$-fold cover of a fundamental domain for the action of
$\Gamma$ on $V_n(F_{S})^{(\sigma)}_X$.

To construct $\FF(X)$, we recall the reduction theory developed by
Springer \cite{Sp}. Let $P$ be a minimal parabolic subgroup of
$G_n$. Let $T$ be a maximal split torus of $G_n$ contained in $P$, and
let $N$ be the unipotent radical of $P$. Finally, let $\Delta$ denote
a set of simple roots. That is, $\Delta$ is a basis for the set of
positive roots defined by $P$.

We use the following coordinates for $T$ and $\Delta$:
\begin{equation}
\begin{array}{lll}
n=2:&T=1;&\\[.1in]
n=3:
&T=\Bigl\{\left(
\footnotesize{\begin{array}{cc}s_1^{-1}&{}\\{}&\!\!\!\!s_1\end{array}}
\right)\Bigr\},&\Delta = \{s_1^2\};\\[.2in]
n=4:
&T=\left\{
\left(
\footnotesize{\begin{array}{cc}s_1^{-1}&{}\\{}&\!\!\!\!s_1\end{array}}
\right),\biggl(
\footnotesize{\begin{array}{ccc}
s_2^{-2}s_3^{-1}&{}&{}\\{}&\!\!\!\!\!\!s_2s_3^{-1}&{}\\{}&{}&\!\!\!\!\!\!s_2s_3^2
\end{array}}
\biggr)\right\},&\Delta = \{s_1^2,s_2^3,s_3^3\};\\[.2in]
n=5:
&T=\left\{
\left(
\footnotesize{\begin{array}{cccc}
\!\!\!s_1^{-3}s_2^{-2}s_3^{-1}&{}&{}&{}\\[.01in]
{}&\!\!\!\!\!\!\!\!\!\!\!\!\!\!\!\!\!\!\!\!\!\!\!\!\!\!\!\!\!\!\!\!\!\!\!\!s_1^{1}s_2^{-2}s_3^{-1}&{}&{}\\[.01in]
{}&{}&\!\!\!\!\!\!\!\!\!\!\!\!\!\!\!\!\!\!\!\!\!\!\!\!\!\!\!\!\!\!\!\!\!\!\!s_1^{1}s_2^{2}s_3^{-1}&{}\\[.01in]
{}&{}&{}&\!\!\!\!\!\!\!\!\!\!\!\!\!\!\!\!\!\!\!\!\!\!\!\!\!\!\!\!\!\!\!\!\!\!\!\!\!\!\!\!s_1^{1}s_2^{2}s_3^{3}
\end{array}}
\!\!\!\!\!\!\!\!\!\!\!\right),
\left(
\footnotesize{\begin{array}{ccccc}
\!\!\!\!\!s_4^{-4}s_5^{-3}s_6^{-2}s_7^{-1}&{}&{}&{}&{}\\[.01in]
{}&\!\!\!\!\!\!\!\!\!\!\!\!\!\!\!\!\!\!\!\!\!\!\!\!\!\!\!\!\!\!\!\!\!\!\!\!\!\!\!\!\!\!s_4^{1}s_5^{-3}s_6^{-2}s_7^{-1}&{}&{}&{}\\[.01in]
{}&{}&\!\!\!\!\!\!\!\!\!\!\!\!\!\!\!\!\!\!\!\!\!\!\!\!\!\!\!\!\!\!\!\!\!\!\!\!\!\!\!\!\!\!s_4^{1}s_5^{2}s_6^{-2}s_7^{-1}&{}&{}\\[.01in]
{}&{}&{}&\!\!\!\!\!\!\!\!\!\!\!\!\!\!\!\!\!\!\!\!\!\!\!\!\!\!\!\!\!\!\!\!\!\!\!\!\!\!\!\!\!\!s_4^{1}s_5^{2}s_6^{3}s_7^{-1}&{}\\[.01in]
{}&{}&{}&{}&\!\!\!\!\!\!\!\!\!\!\!\!\!\!\!\!\!\!\!\!\!\!\!\!\!\!\!\!\!\!\!\!\!\!\!\!\!\!\!\!\!\!s_4^{1}s_5^{2}s_6^{3}s_7^{4}
\end{array}}
\!\!\!\!\!\!\!\!\!\!\right)\right\},&\Delta = \{s_1^4,s_2^4,s_3^4,s_4^5,s_5^5,s_6^5,s_7^5\}.
\end{array}
\end{equation}

Set $S_\infty$ to be $M_\infty$ when $F$ is a number field and to be
any (fixed) nonempty subset of $S$ when $F$ is a function field. For
any positive constants $c$ and $c'$, define:
\begin{eqnarray*}
 T(c) &=& \{t=(t_\fP)_{\fP\in S_\infty}\in T(F_{S_\infty}) : |\alpha(t)| \geq c,\forall \alpha\in \Delta\},\\
 T(c,c') &=& \{t\in T(c) : ||\frac{|t_w|_w^{1/\deg w}}{|t_{w'}|_{w'}^{1/\deg w'}}||_\infty \leq c',\forall w,w'\in S_\infty\},
\end{eqnarray*}
where we identify $T$ with $\bG_m^r$, view
$|t_w|_w^{1/\deg w}/|t_{w'}|_{w'}^{1/\deg w'}$ as an element of $\R^r$ and where
$||\cdot||_\infty$ denotes the supremum norm on $\R^r$. Then by
\cite[Remark 2.2]{Sp} there exist positive real numbers $c,c'$, a
compact subset $N'\subset N(F_{S_\infty})$, and a compact subgroup
$K'$ of $G_n(F_S)$ such that
\begin{equation}\label{eq:Spred}
G_n(F_S)^1 = G_n(\cO_S)N'T(c,c')K'.
\end{equation}
We remark that in \cite{Sp}, the set $T(c)$ is defined by $|\alpha(t)|
\leq c.$ Since we want a fundamental domain for the left action of
$G_n(\cO)$ on $G_n(F_S)$, we need to apply inverses to the results of
\cite{Sp}. The extra parameter $c'$ comes from computing
$T(\cO_S)\backslash T(c)$ using the fact that the image of
$\cO_S^\times$ in $\R^{|S_\infty|}$ under the map $t\mapsto (\log
|t_v|_v)$ is a lattice of rank $|S_\infty|-1$ in the hyperplane
$H:x_1+\cdots+x_{|S_\infty|} = 0$ with compact quotient. The scaling factor $1/\deg w$ is unnecessary in the function field case. In the number field case, it guarantees that when $t$ is viewed in $F_{S_\infty}^r = \R^{rd}$, every two coordinates differ multiplicatively by $O(1)$.

The subset $N'T(c,c')K'$ is called a {\it Siegel domain} $\D_0$. Since
$\Gamma$ is commensurable with $G_n(\cO_S)$, there exists a finite set
of elements $g_1,\ldots,g_k\in G_n(F)$ such that the union
$\cup_{i=1}^k g_i\D_0$ of translates of $\D_0$ contains a fundamental
domain $\Omega$ for the action of $\Gamma$ on $G_n(F_S)^1.$ To obtain
$\FF(X)$ or $\FF$, we use the center of $G_n$ to apply a scaling as
follows. We write $g(\lambda)$ for the element of the center of $G_n$
that scales every coordinate of $V_n$ by $\lambda$. When $F$ is a
number field, there is an embedding $\iota$ of $\R^+$ into
$G_n(F_{S})$ sending $\lambda\in\R^+$ to the adele $\iota(\lambda)$
that is $g(\lambda^{\kappa})$ at every infinite place and is $1$ at
every finite place in $S$ where $\kappa$ is a positive real constant
chosen so that the $S$-norm $|\chi(\iota(\lambda))|$ is $\lambda$. Let
$\Lambda_X$ be the image of this embedding of the interval
$(0,(X/|\Delta_n(v_\sigma)|)^{1/2}].$ Then we may take $\FF(X)$ to be
  $\Lambda_X\Omega$. We let $\Lambda$ denote the entire image of
  $\iota$ and set $\FF = \Lambda\Omega$.

When $F$ is a function field of characteristic $p$, one could let
$\Lambda_X$ be an arbitrary (fixed) element of
$G_n(F_{S})^{(\sigma)}_X$ and take $\Lambda_X\Omega$ to be $\FF(X)$. We
will use this construction when we later compute the various volumes
to obtain the main term of the estimate. However in order for the
shape of the fundamental domain to remain the same as $X$ grows, for
the purpose of a controllable error term, we make the following
modification. We fix an embedding of $\F_p(u)$ into $F$. Precomposing
it with the map $\Z\rightarrow \F_p(u)$ sending $m$ to $u^m$, and
postcomposing it with the embedding of $F$ into $F_{S}$ that is the
natural embedding at all the places of $S_\infty$ and is the constant
$1$ at all other places, give an embedding $\Z\rightarrow F_S$. Next
we postcompose it with the embedding of $F_{S}$ into $G_n(F_S)$
sending $\lambda$ to $g(\lambda)$ as described above. Denote by
$\iota$ the resulting embedding $\Z\rightarrow G_n(F_S)$. Let
$\Lambda$ denote the image of $\iota$. A finite union of right
translates of the set $\Lambda\Omega$ then forms a fundamental domain
$\FF$ for the action of $\Gamma$ on $G_n(F_S)$. Note that one may
combine these right translates with the compact subgroup $K'$. We set
$\FF(X)$ to be the intersection of $\FF$ with
$G_n(F_S)^{(\sigma)}_X$. Of course, we only consider $X$ when this set
is nonempty, for otherwise there are no orbits such that the $S$-norm
of the $\Delta_n$ invariant is $X$.

We summarize the above construction in the following theorem.

\begin{theorem}\label{thm:fundG}
 There exists a subset $\D$ of $G_n(F_S)$ of the form $\Lambda
 N'T(c,c')K''$ where $\Lambda$ is a subset of the center of
 $G_n(F_S)$, where $N'$ and $T(c,c')$ are as above, and where $K''$ is
 a finite union of right translates of a compact subgroup $K'$ of
 $G_n(F_S)$, such that a fundamental domain for the left action of
 $\Gamma $ on $G_n(F_{S})$ is contained in a finite union of
 translates $g_i\D$, with $g_i\in G_n(\cO_S)$, for $i=1,\ldots,k$.
\end{theorem}

For each $i=1,\ldots,k$, let $\FF_{i}$ denote $\FF\cap g_i\D.$
Then $\FF$ is the disjoint union of $\FF_{1},\ldots,\FF_{k}.$

\subsection{Averaging}

Theorem~\ref{thmfunddomain} implies that we have
\begin{equation*}
N(\L^{(\sigma)},\Gamma;X)=\frac{1}{\#\Aut(\sigma)}\#\bigl\{\FF(X)\cdot
v_\sigma\cap\L^{\gen}\bigr\}.
\end{equation*}
When $F$ is a number field, we define $G_0\subset G_n(F_S)$ to be some
fixed product of nonempty open bounded semialgebraic sets $G_0'\subset
G_n(F_{M_\infty})$ and nonempty open compact sets $G_0''\subset
G_n(F_{S\backslash M_\infty})$. When $F$ is a function field, we set
$G_0$ to be some fixed nonempty open compact subset of
$G_n(F_{S})$. Suppose further that $|\chi(g)|>1$ for every $g\in G_0$.

We have the
following equalities used in~\cite{dodpf}:

\begin{equation}\label{eqavg}
\begin{array}{rcl}
N(\L^{(\sigma)},\Gamma;X)
&=&\displaystyle\frac{1}{\Vol(G_0)\#\Aut(\sigma)}\int_{g\in G_0}\#\bigl\{(\FF g\cdot v_\sigma)_X\cap\L^{\gen}\bigr\}dg\\[.2in]
&=&\displaystyle\frac{1}{\Vol(G_0)\#\Aut(\sigma)}\int_{g\in\FF}\#\bigl\{(gG_0\cdot v_\sigma)_X\cap\L^{\gen}\bigr\}dg\\[.2in]
&=&\displaystyle\frac{1}{\Vol(G_0)\#\Aut(\sigma)}\sum_{i=1}^k\int_{g\in\FF_{i}}\#\bigl\{(gG_0\cdot v_\sigma)_X\cap\L^{\gen}\bigr\}dg,
\end{array}
\end{equation}
where $dg$ is any Haar-measure on $G_n(F_S)$ and the volume of
$G_0$ is taken with respect to $dg$.

The action of $g_i^{-1}$ on $V_n(F)$ takes $\L$ to a different lattice
and preserves generic elements and the $S$-norms of their discriminants. Then we have
\begin{equation}\label{eqdivide}
\int_{g\in\FF_{i}}\#\bigl\{(gG_0\cdot v_\sigma)_X\cap\L^{\gen}\bigr\}dg=
\int_{g\in g_i^{-1}\FF_{i}}\#\bigl\{(gG_0\cdot v_\sigma)_X\cap g_i^{-1}\L^{\gen}\bigr\}dg.
\end{equation}
Since $g_i^{-1}\L$ is also a lattice in $V_n(F_S)$, the $S$-norms of
nonzero coefficients of elements in $g_i^{-1}\L$ are uniformly bounded
from below by a positive constant $c_i>0$.

Let $\var$ denote the set of coefficients of $V$. For $\alpha\in\var$,
let $w(\alpha)$ denote the quantity by which an element of $\Lambda T$
scales $\alpha$. Then $w(\alpha)$ is a monomial in $\lambda$ and the
$s_i$. It is easy to see that the exponent of $\lambda$ appearing in
the weight of each $\alpha\in\var$ is the same. We define a partial
order $\leq$ on $\var$, where we set $\alpha\leq\beta$ if the weight
$w(\beta\alpha^{-1})$, when viewed as a character of $T$ is positive;
that is, it is a product of nonnegative powers of the $s_i$. In all
our cases, $\var$ contains an element $\alpha_0$ with
$\alpha_0\leq\alpha$ for all $\alpha\in\var$. For $V_3$ we have $\alpha_0=a$, the
coefficient of $x^3$, for $V_4$ we have $\alpha_0=a_{11}$, and for
$V_5$ we have $\alpha=a_{12}$. Since our representations are not
trivial, we see that the powers of every $s_i$ in
$w(\alpha_0)$ is negative. Let $\FF_{i,X}'$ denote the set of elements
$g\in \FF_{i}$ such that $(g_i^{-1}gG_0\cdot v_\sigma)_X$ contains
elements whose $\alpha_0$-coefficients have $S$-norm at least
$c_i$. It follows that if $g\in\FF_{i}\backslash\FF_{i,X}'$, then
every element of $(g_i^{-1}gG_0\cdot v_\sigma)_X\cap g_i^{-1}\L$ has
$\alpha_0$-coefficient equal to~0.

We call the set $g_i^{-1}\FF_{i,X}'G_0\cdot v_\sigma$ ``the main body''
and the set $g_i^{-1}(\FF_{i}\backslash\FF_{i,X}')G_0\cdot v_\sigma$
``the cuspidal region''.  In the \S\ref{sec:mainnongensmall} and
\S\ref{sec:cuspgensmall}, we prove respectively that the number of
non-generic elements in the main body is negligible and that the
number of generic elements in the cuspidal region is negligible.

\subsection{The number of generic elements in the cusp is negligible}\label{sec:cuspgensmall}

Let $L$ be some fixed lattice in $V_n(F_S)$ that is commensurable with
$V_n(\cO_S)$. We will apply the result of this section when
$L=g_i^{-1}\L$. The $S$-norms of the nonzero coefficients of elements
in $gG_0\cdot v_\sigma\cap L$ for $g\in\D$ are uniformly bounded from
below, say by some constant $c_0>0$. Write $\D_X=\D\cap
G_n(F_S)_X^{(\sigma)}$. and let $\D_X'$ denote the set of elements
$g\in\D_X$ such that $gG_0\cdot v_\sigma$ contains elements whose
$\alpha_0$-coefficients have $S$-norm at least $c_0$. Since
$|\chi(g)|$, for any $g\in G_0$, is bounded below and above by an
absolute constant, it suffices to prove the following theorem.
\begin{theorem}
We have
\begin{equation}\label{eqcuspcutoff}
\int_{g\in\D_X\backslash\D_X'}\#\{gG_0\cdot v_\sigma\cap L^\gen\}dg=O(X^{1-\epsilon_n})
\end{equation}
where $\epsilon_n=1$ when $n=2$ or $3$ and $\epsilon_n=1/d_n$ when $n=4$ or $5$; here $d_n$ is the dimension of $V_n$.
\end{theorem}
\begin{proof}
When $n=2$ or $n=3$, the theorem is immediate because every $v\in
gG_0\cdot v_\sigma\cap L$, with $g\in \D_X\backslash\D_X'$, satisfies
$\alpha_0(v)=0$, which implies that $v$ is not generic. When $n=4$ or
$n=5$, we proceed as in \cite{dodpf}.

For a subset $U\subset\var$, let $L(U)$ denote the set of elements of
$L$ such that: for every $\alpha\in U$, $\alpha(v)=0$ for every $v\in
L(U)$; and for every $\alpha\notin U$, there exists $v\in L(U)$ such
that $\alpha(v)\neq 0$. There are characters $\delta_n$ of the torus
(not involving the center) such that the measure
$\delta_n(s)d^\times\lambda dnd^\times s dk$ is a (left-invariant)
Haar measure on $G_n(F_S)$, where $dk$ is the measure on $K''$
obtained from translating the Haar-measure on $K'$, $dn$ is a
Haar-measure on $N(F_{S_\infty})$ and $d^\times s$ denotes $\prod
ds_i/s_i$. When $F$ is a number field, we have the embedding $\iota$
of $\R^+$ into the center of $G_n(F_S)$ and $d^\times\lambda$ is the
push forward of the measure $d\lambda/\lambda$ on $\R^+$. When $F$ is
a function field, then $d^\times\lambda$ is $1$ (i.e., it does not
appear in the Haar measure). The characters $\delta_n(s)$ are given
by:
\begin{eqnarray*}
 \delta_2(s) &=& 1,\\
 \delta_3(s) &=& s_1^{-2},\\
 \delta_4(s) &=& s_1^{-2}s_2^{-6}s_3^{-6},\\
 \delta_5(s) &=& s_1^{-8}s_2^{-12}s_3^{-8}s_4^{-20}s_5^{-30}s_6^{-30}s_7^{-20}.
\end{eqnarray*}

Write $\Lambda'_X$ for $\Lambda_X$ when $F$ is a number field and for
$\iota(m)$ when $F$ is a function field where $m$ is the integer such
that $X/|\Delta_n(v_\sigma)|$ lies between $|\iota(m)|$ and
$|\iota(m+1)|$. By definition of $T$ and $N$, we see that there exists
a compact subset $N''\subset N(F_{S_\infty})$ such that for any $t\in
T(c)$ and any $n\in N'$, $t^{-1}nt\in N''$. Hence any element of
$\D_X$ has the form $gk$ where $g\in\Lambda'_X T(c,c')$ and $k$
belongs to the compact set $N''K'.$ It suffices to prove the estimate
\begin{equation}\label{eqcutcusp}
\int_{g\in\Lambda'_X T(c,c')}\#\{gkG_0\cdot v_\sigma\cap
L(U)^\gen\}|\delta_n(s)|d^\times s d^\times\lambda \ll X^{1-1/d_n},
\end{equation}
for any $k\in N''K'$ for all sets $U\subset\var$ containing $\alpha_0$, where the implied constant does not depend on $k$.  As in the proof of
\cite[Lemma~11]{dodpf} it suffices to prove the theorem for sets $U$
that are closed under the partial order $\leq$, meaning that if
$\beta\in U$ and $\alpha\leq \beta$, then $\alpha\in U$.  Also note
for $g=\lambda s \in\Lambda'_X T'$, the set $gkG_0\cdot v_\sigma\cap
L(U)$ is empty unless $w(\alpha)\gg 1$ for all $\alpha\notin U$.

To count the number of lattice points, we use the following result
from the geometry of numbers.

\begin{proposition}\label{davgen}
Let $E$ be $\R$ $($resp.\ $F_S)$ if $F$ is a number field
$($resp. function field and $S$ is a nonempty set of places$)$. Let
$m$ be a positive integer. Let $B$ be an open bounded and
semialgebraic subset of $E^m$ $($resp.\ an open compact subset of
$E^m)$. Let $K$ be any subset of $\GL_m(E)$ $($resp.\ an open compact
subset of $\GL_m(E))$. Let $c$ be a real constant. Let $L$ be a
lattice in $E^m$. Then for any $k\in K$ and any $t =
\diag(t_1,\ldots,t_m)\in\GL_m(E)$ with the additional condition, when
$F$ is a function field, that $|t_i|_v^{1/\deg v}/|t_i|_{v'}^{1/\deg v'}<c$ for any $v,v'\in
S$,
\begin{equation}\label{eq:daven}
 \#\{tkB \cap L\} = \Vol_{L}(tkB) + O(\Vol(\proj(tkB))),
\end{equation}
where $\Vol$ is some fixed volume measure on $E^m$, $\Vol_{L}$ is a constant multiple of $\Vol$ such that $E^m/L$
has volume $1$, and $\Vol(\proj(tkB))$ denotes the greatest
$\ell$-dimensional volume of any projection of $tkB$ onto a coordinate
subspace obtained by equating $m-\ell$ coordinates to zero, where $d$
takes all values from $1$ to $m-1$.  The implied constant in the
second summand depends only on $E$, $m$, $B$, $c$, $L$ and in the
function field case, also on $K$.
\end{proposition}

When $F$ is a number field, this result follows from Davenport
\cite{Lipschitz} where the dependency on $B$ is through the number and
the degrees of the polynomial inequalities that define $B$, which stay
the same when a linear change of variable is applied. When $F$ is a
function field, see Proposition \ref{prop:skewexp} in Appendix A.

Using Proposition~\ref{davgen} with $B = G_0\cdot v_\sigma$ and $K =
N''K'$, we have for any $k\in K$ and for any nonnegative real numbers
$k_\alpha$,
\begin{eqnarray*}
&&\displaystyle\int_{g\in\Lambda'_X T(c,c')}\#\{gkG_0\cdot v_\sigma\cap
L(U)^\gen\}d^\times\lambda d^\times s \\
&\ll&
\displaystyle\int_{g\in\Lambda'_X
  T(c,c')}\bigl(\prod_{\alpha\not\in U}|w(\alpha)|\bigr)
|\delta_n(s)|d^\times sd^\times\lambda\\
&\ll&\displaystyle\int_{g\in\Lambda'_X
  T(c,c')}\bigl(\prod_{\alpha\not\in U}|w(\alpha)|\bigr)
\bigl(\prod_{\alpha\in
  \min(U)}|w(\alpha)|^{k_\alpha}\bigr)|\delta_n(s)|d^\times sd^\times\lambda\\
&\ll&X\displaystyle\int_{g\in\Lambda'_X T'}
\bigl(\prod_{\alpha\in U}|w(\alpha)^{-1}|\bigr)
\bigl(\prod_{\alpha\in
  \min(U)}|w(\alpha)|^{k_\alpha}\bigr)|\delta_n(s)|d^\times sd^\times\lambda.
\end{eqnarray*}
Therefore, to prove \eqref{eqcutcusp} for a set $U$, it suffices to
find nonnegative real numbers $k_\alpha$ for $\alpha\notin U$ such
that $\sum k_\alpha< \#U$ and the exponents of each $s_i$ in
$\bigl(\prod_{\alpha\in U}w(\alpha)^{-1}\bigr)\bigl(\prod_{\alpha\in
  \min(U)}w(\alpha)^{k_\alpha}\bigr)\delta_n(s)$ are negative. The
condition $\sum k_\alpha< \#U$ implies that the exponent of $\lambda$
is also negative.

We now prove the theorem for $n=4$.  Every element $(A,B)\in\L$ with
$a_{11}=b_{11}=0$ (resp.\ $a_{11}=a_{12}=a_{13}=0$ or
$a_{11}=a_{12}=a_{22}=0$) is reducible because one of the four points
in $\P^2(\bar{F})$ corresponding to $(A,B)$ is rational
(resp.\ $\det(A)=0$; thus the cubic resolvent of $(A,B)$ is
reducible).  Thus it is enough to consider sets $U$ equal to
$\{a_{11}\}$ and $\{a_{11},a_{12}\}$. For $U=\{a_{11}\}$, we can take
the weight $1$: that is $k_\alpha=0$ for all $\alpha$. For $U =
\{a_{11},a_{12}\}$, we take the weight $w(a_{12})$: that is $k_\alpha
= 1$ for $\alpha = a_{12}$ and $0$ for all other $\alpha$. The proof
of the theorem when $n=5$ follows in an identical fashion from
\cite[Lemma~10]{dodpf} and \cite[Table 1]{dodpf}.
\end{proof}

\subsection{The main term}

With notation as above, we have the following lemma.
\begin{lemma}
For $g\in g_i^{-1}\FF_{i,X}'$, we have
\begin{equation*}
\#\{(gG_0\cdot v_\sigma)_X\cap L\}=\Vol_{L}\bigl((gG_0\cdot v_\sigma)_X\bigr)+
O(X|w(\alpha_0)|^{-1/d}),
\end{equation*}
where the volume is computed with respect to the Haar-measure of
$V_n(F_S)$ normalized so that $L$ has covolume 1 in $V_n(F_S)$ and
where $d$ is the degree of $F$ over $\Q$ when $F$ is a number field
and is $1$ when $F$ is a function field.
\end{lemma}

\begin{proof}
  Since the projection of $gG_0\cdot v_\sigma$ onto
  $V_n(F_{S\backslash S_\infty})$ is open compact, we may replace $L$
  by a union of finitely many lattices in $V_n(F_{S_\infty})$ and
  applying Proposition \ref{davgen} to each of these lattices gives
  the main term. The error term is then the maximum of the volumes of
  the images onto coordinate hyperplanes of $V_n(F_{S_\infty})$. Since
  $\alpha_0$ has minimal weight, we see that to achieve the maximum of
  the volumes (up to a bounded constant), we must include all the
  coordinate hyperplanes away from $\alpha_0$. We are thus done in the
  function field case by Proposition~\ref{davgen}. When $F$ is a
  number field, $V_n(F_{S_\infty})$ is a real vector space of
  dimension $\dim(V_n)d$. By the definition of $T(c,c')$, the
  projection of $(gG_0\cdot v_0)_X$ onto the $d$ $\R$-coordinates
  corresponding to $\alpha_0$ all have sizes bounded above and below
  by an absolute constant times $|w(\alpha_0)|^{1/d}$, and the maximum
  of the volumes of all projections, up to a bounded constant, is then
  achieved by taking one of these $d$ coordinates and projecting onto
  the coordinate hyperplanes corresponding to all the other
  coordinates.
\end{proof}

Integrating both sides of the equation in the above lemma over
$g\in\cup_i g_i^{-1}\FF_{i,X}'$, we obtain
\begin{equation}\label{appdavlem}
\begin{array}{rcl}
\displaystyle\int_{g\in g_i^{-1}\FF_{i,X}'}\#\{(gG_0\cdot v_\sigma)_X\cap g_i^{-1}\L\}&=&
\displaystyle\int_{g\in g_i^{-1}\FF_{i,X}'}\Vol_{g_i^{-1}\L}\bigl((gG_0\cdot
v_\sigma)_{X}\bigr)dg+O(X^{1-1/\epsilon'_n})\\[.2in]
&=&
\displaystyle\int_{g\in\FF_{i,X}'}\Vol_{\L}\bigl((gG_0\cdot
v_\sigma)_{X}\bigr)dg+O(X^{1-1/\epsilon'_n})\\[.2in]
&=&
\displaystyle\int_{g\in\FF_{i}}\Vol_{\L}\bigl((gG_0\cdot
v_\sigma)_{X}\bigr)dg+O(X^{1-1/\epsilon'_n}),
\end{array}
\end{equation}
where $\epsilon'_2=2d$, $\epsilon'_3=6d$, $\epsilon'_4=12 d$, and
$\epsilon'_5=40d$. The first equality follows from an explicit
computation for the integral of $|w(\alpha_0)|^{-1/d}$. The second
equality follows since $g_i\in G_n(F_S)^1$ is measure preserving. The
last equality follows from an element computation of the volume of
$\FF_{i}\backslash\FF_{i,X}'$.

\subsection{The number of nongeneric elements in the main body is
  negligible}\label{sec:mainnongensmall}

Let $\L^\ngen$ denote the set
of elements in $\L$ that are not generic. In this subsection, we prove
the following theorem.
\begin{theorem}
With notations as above, we have
\begin{equation}\label{eqngenmb}
\displaystyle\int_{\FF_{i,X}'}\#\{(gG_0\cdot v_\sigma)_X\cap\L^{\ngen}\}dg=o(X).
\end{equation}
\end{theorem}
\begin{proof}
Let $\tau$ be a local splitting type corresponding to some conjugacy
class of $S_n$. In particular, $\tau$ is an unramified splitting
type. For a prime ideal~$\fP$ of $\cO_S$, let $\L^{\fP,\neq(\tau)}$
denote the set of elements in $v\in\L$ that do not have splitting type
$\tau$ at $\fP$.  Note that if $v\in\L$ is generic, then $v$ must have
splitting type $\tau$ at some (positive proportion of) primes.
Therefore, we have
\begin{equation*}
\L^{\ngen}\subset \bigcup_\tau\Bigl(\bigcap_\fP\L^{\fP,\neq(\tau)}\Bigr).
\end{equation*}


The set $\L^{\fP,\neq (\tau)}$ is contained in the inverse image under the
reduction modulo $\fP$ map of the set $V_n(\F^{\neq (\tau)}_\fP)$
consisting of elements in $V_n(\F_\fP)$ that do not have splitting
type $(\tau)$. Denote by $\mu_\tau(\fP)$ the $\fP$-adic density of
$\L^{\fP,(\neq\tau)}$, that is, $\mu_\tau(\fP) \leq \#V_n(\F^{\neq
  (\tau)}_\fP)/\#V_n(\F_\fP)$. Then as $N\fP\to\infty$,
$\mu_\tau(\fP)$ is bounded above by a constant $c_\tau:=1/\#\Aut(\tau)$ which
is strictly between $0$ and $1$; this follows from \cite[Lemma
  18]{BST}, \cite[Lemma 21]{BIII}, \cite[Lemma 20]{BIV}
(see~\cite{Bmass} for the definition of $\#\Aut(\tau)$).

Let $Y>0$ be fixed. Then from the results of the last subsection we
see that
\begin{equation*}
\displaystyle\int_{\FF'_{i,X}}\#\{gG_0\cdot v_\sigma\cap
\bigcap_{N(\fP)<Y}\L^{\fP,\neq(\tau)}\}dg\ll X\prod_{N(\fP)<Y}\mu_\tau(\fP),
\end{equation*}
where the implied constant only depends on $X$. Letting $Y$ tend to
infinity, we obtain the result.
\end{proof}

\noindent {\bf Remark:} We note that an application of the Selberg
Sieve used exactly as in \cite{AJ} improves the right hand side of
\eqref{eqngenmb} to $O(X^{1-1/(5dd_n)+\epsilon})$, where $d_n$ is the
dimension of $V_n$.

\subsection{Conclusion}
By combining \eqref{eqavg}, \eqref{eqdivide},
\eqref{eqcuspcutoff}, \eqref{appdavlem}, and
\eqref{eqngenmb}, we obtain
\begin{equation}\label{eqfinalcountnf}
\begin{array}{rcl}
N(\L^{(\sigma)},\Gamma;X)&=&\displaystyle\frac{1}{\#\Aut(\sigma)\Vol(G_0)}\int_{g\in\FF}\Vol_\L
\bigl((gG_0\cdot v_\sigma)_X\bigr)dg + o(X)\\[.2in]
&=&\displaystyle\frac{1}{\#\Aut(\sigma)\Vol(G_0)}\int_{g\in G_0}\Vol_\L
\bigl((\FF g\cdot v_\sigma)_X\bigr)dg + o(X)\\[.2in]
&=&\displaystyle\frac{1}{\#\Aut(\sigma)}\Vol_\L
\bigl(\FF(X)\cdot v_\sigma\bigr)+o(X).
\end{array}
\end{equation}

Therefore, we have the following theorem.
\begin{theorem}\label{thmaincountgon}
Let $\L\subset V_n(F)$ be a sublattice commensurable with $V_n(\cO_S)$
and let $\Gamma\subset G_n(F)$ be a subgroup commensurable with
$G_n(\cO_S)$ that preserves $\L$. Fix an $S$-specification $\sigma$
and $v_\sigma\in V_n(F_S)^{(\sigma)}$. Let $\FF(X)$ be a fundamental
domain for the action of $\Gamma$ on $G_n(F_S)^{(\sigma)}_X$
constructed in {\rm $\S\ref{sec:fund}$}. Then
\begin{equation*}
N(\L^{(\sigma)},\Gamma;X)=\frac{1}{\#\Aut(\sigma)}\Vol_\L
\bigl(\FF(X)\cdot v_\sigma\bigr)+o(X).
\end{equation*}
\end{theorem}

\noindent {\bf Remark:} As in the previous subsection, the error term
of $o(X)$ in Theorem \ref{thmaincountgon} can be improved to
$O_\L(X^{1-1/(5dd_n)})$.

\section{Congruence conditions and a squarefree sieve}\label{sec:sieve}

Let $F$ be a fixed global field and let $S$ again be a nonempty finite
set of places of $F$ that contains all the archimedean places. Fix
lattices $\Gamma\subset G_n(F)$ and $\L\subset V_n(F)$ that are
commensurable with $G_n(\cO_S)$ and $V_n(\cO_S)$, respectively, and
such that $\Gamma$ preserves $\L$. For each prime ideal $\fP\subset
\cO_S$, let $Z_\fP$ be a compact subset of $V_n(F_\fP)^{\Delta\neq 0}$
whose boundary has measure 0 and is preserved by $\Gamma_\fP$, where
$\Gamma_\fP$ is the closure of $\Gamma$ in $G_n(F_\fP)$. To such a
collection $(Z_\fP)_\fP$, we associate the set $Z\subset\L$ consisting
of elements $v\in\L$ such that when viewed as an element of
$V_n(F_\fP)$, $v\in Z_\fP$ for all primes $\fP$ of $\cO_S$. Such a set
$Z$ is said to be {\it defined by congruence conditions}. We further
say that $Z$ is {\it large} if for all but finitely many primes $\fP$
in $\cO_S$, the set $Z_\fP$ contains all the elements $v\in
V_n(\cO_\fP)$ such that $\theta_{\cO_\fP}(v)$ is a maximal rank $n$
ring over $\cO_\fP$ and $\theta_{F_\fP}(v)$ is a degree-$n$ \'{e}tale
extension that is either unramified or have splitting type $(1^2\tau)$
where $\tau$ is an unramified splitting type of dimension $n-2$. We
say $v\in V_n(\cO_\fP)$ have {\it extra ramification} if
$\theta_{F_\fP}(v)$ is ramified and does not have splitting type
$(1^2\tau)$ where $\tau$ is an unramified splitting type of dimension
$n-2$. When the characteristic of $F$ is not $2$, being non-maximal or
having extra ramification is equivalent to $\fP^2$ dividing the
discriminant. When the characteristic of $F$ is $2$, we do not use
squarefree discriminants to define largeness since the discriminant is
always a square (even as a polynomial)! In this section, we prove the
following theorem.
\begin{theorem}\label{thsieve}
Let $Z$ be a large $\Gamma$-invariant set as above defined
via the local conditions $(Z_\fP)_\fP$. Then
\begin{equation}\label{eqaftersieve}
N(Z^{(\sigma)},\Gamma;X)=\frac{1}{\#\Aut(\sigma)}\Vol_\L(\FF(X)\cdot
v_\sigma)\prod_{\fP\subset\cO_S}\Vol(Z_\fP)+o(X),
\end{equation}
where the volume of $Z_\fP$ is computed with respect to the
Haar-measure on $V_n(\cO_\fP)$ normalized so that the total measure of
$\L_\fP$ is $1$.
\end{theorem}

When $Z_\fP=V(\cO_\fP)$ for all but finitely many primes of $\cO_S$,
and is defined via finitely many congruence conditions at the other
finitely many primes, then $Z$ is a $\Gamma$-invariant lattice and the
results of Section \ref{sec:count} apply, yielding
\eqref{eqaftersieve}. When $Z$ is defined via infinitely many
congruence conditions, then we apply a simple sieve to obtain the
result. The key ingredient in this sieve is the following tail
estimate.

\begin{theorem}\label{thunif}
For a prime $\fP\subset \cO_S$, let $\W_\fP$ denote the set of
elements in $V_n(\cO_S)$ that are non-maximal or have extra
ramification at $\fP$. Then
\begin{equation}\label{equnif}
\sum_{N\fP>M}N(\W_\fP,G_n(\cO_S);X)=O(X/(M\log M))+o(X),
\end{equation}
where the implied constants are independent of $M$ and $X$.
\end{theorem}

When $F=\Q$, this was proved in \cite[\S4.2]{Bgeosieve}. Some mild
arguments are needed to generalize it to arbitrary global fields. The
rest of this section is dedicated to this generalization.

We write $\W_\fP$ as the union $\W_\fP^{(1)}\cup\W_\fP^{(2)}$. The set
$\W_\fP^{(1)}$ consists of elements that are non-maximal but does not
have extra ramification and the set $\W_\fP^{(2)}$ consists of
elements that have extra ramification. When the characteristic of $F$
is not $2$, $\W_\fP^{(2)}$ consists of elements whose discriminants
are divisible by $\fP^2$ for (mod $\fP$) reasons and we say their
discriminants are strongly divisible by $\fP^2$ while for elements of
$\W_\fP^{(1)}$, we say their discriminants are weakly divisible
by~$\fP^2$.

For elements with extra ramification, there exists a subscheme $Y$ of
$\A_{\cO_S}^{d_n}$ of codimension $2$ such that
$$\W_\fP^{(2)}\subset \{v\in V_n(\cO_S): v(\mbox{mod }\fP) \in
Y(k(\fP))\}.$$ Recalling the shape of the fundamental domain for the
action of $G_n(\cO_S)$ on $V_n(F_S)$ in Section \ref{sec:fund}, we see
that in order to obtain the estimate \eqref{equnif} for
$\W_\fP^{(2)}$, it suffices to obtain the following estimate.

\begin{theorem}\label{thm:Th7F}
 Let $B$ be an open bounded region in $V_n(F_S)$, $Y$ be a closed
 subscheme of $\A_{\cO_S}^{d_n}$ of codimension $k\geq 2$ and let $M$
 be a positive real number. Let $r$ be a positive real number when $F$
 is a number field or an integer when $F$ is a function field of
 characteristic $p$ along with a fixed embedding of $\F_p(u)$ into
 $F$. Let $\iota$ denote the embedding of $\R^+$ or $\Z$ in $G_n(F_S)$
 described in Section \ref{sec:fund}. Then we have
\begin{eqnarray}
 \nonumber&&\#\{a\in \iota(r)B\cap V_n(\cO_S): a (\mbox{mod }\fP) \in Y(k(\fP))\mbox{ for some prime }\fP\notin S\mbox{ with }N\fP>M\} \\
 \label{eq:estF}&=& O\left(\frac{|\chi(\iota(r))|^2}{M^{k-1}\log M}\right) + o(|\chi(\iota(r))|^2),
\end{eqnarray}
as $|\chi(\iota(r))|$ tends to infinity, where the implied constant depends only on $B$, $Y$ and $F$.
\end{theorem}

\begin{proof}
Let $\lambda\in F_S$ be such that $\iota(r) = g(\lambda)$. That is,
when $F$ is a number field, $\lambda$ is the adele that is $r^\kappa$
at all the infinite places and $1$ elsewhere; and when $F$ is a
function field, $\lambda$ is the adele that is $u^r$ at all the places
of $S_\infty$ and $1$ elsewhere. Since $g(\lambda)$ acts on $V_n$ by
scaling each coordinate by $\lambda$, we see that it scales the volume
of any open bounded region by $|\lambda|^{d_n} =
|\chi(g(\lambda))|^2$. This explains the order of magnitude in
\eqref{eq:estF}.
The remainder of the estimate is then a direct generalization of
\cite[Theorem~3.3]{Bgeosieve}.
\end{proof}

To deal with $\W_\fP^{(1)}$, the key strategy of \cite{Bgeosieve} is
that for any $v\in \W_\fP^{(1)}$, there exists an element $g\in
G_n(\Q)$ such that $\Delta_n(g.v) = \Delta_n(v)/p^2.$ Over a general
global field, such a global element $g$ might not exist due to the
class group. However one can still do it locally. Define $B_\fP^{(1)}$
to be the subset of
$V_n(\cO_\fP)$ consisting of elements that are non-maximal but does not have extra ramification.

Now take any element $v$ of $\W_\fP^{(1)}$. Then its specialization
$v_\fP$ when viewed as an element of $V_n(\cO_\fP)$ is in
$B_\fP^{(1)}$. Since $\theta_{\cO_\fP}(v_\fP)$ is not maximal, there exists $g_\fP\in
G_n(F_\fP)$ such that $g_\fP v_\fP\in V_n(\cO_\fP)$ is maximal and
with $$|\Delta_n(g_\fP v_\fP)|_\fP \leq |\Delta_n(v_\fP)|_\fP/N\fP^2.$$
Consider the adele $g\in G_n(\A_S)$ that equals $g_\fP$ at the place
$\fP$ and is $1$ everywhere else. From the decomposition
\eqref{eqstrongapproximation0}, there exists $\gamma\in\cl_S$, $h\in
G_n(F)$, $h_{\fP'}\in G_n(\cO_{\fP'})$ for any $\fP'\notin S$ such
that $g'=(h_{\fP'})_{\fP'\notin S}\gamma h$ in $G_n(\A_S)$. Set $v'\in
hv\in V_n(F)$. Then $v'$ lies in $\L_\gamma$ with
$$|\Delta_n(v')| \leq C_1\frac{|\Delta_n(v)|}{N\fP^2}$$ where $C_1 =
|\chi(\gamma)|^{-2}$ is a constant depending only on $G_n$ and $F$. We
have therefore defined the following map
$$\pi_\fP:G_n(\cO_S)\backslash \W_{\fP}^{(1)}\rightarrow
\bigcup_{\gamma\in\cl_S} \Gamma_\gamma \backslash \L_\gamma$$ such
that $|\Delta_n(\pi_\fP(v))| = C_1|\Delta_n(v)|/N\fP^2$.  We remark
that the choice of $\gamma$ is unique as it is determined by the field
extension corresponding to $v$ (Theorem~\ref{thmainparam}).

Finally the same argument as in \cite[\S4.2]{Bgeosieve} shows that the
fibers of $\pi_\fP$ have absolutely bounded sizes (in fact, they have
cardinality at most $10$). Hence we conclude that
\begin{equation}\label{eq:main_unif_eq}
N(\W_\fP^{(1)},G_n(\cO_S);X) = O(X/N\fP^2).    
\end{equation}
Summing over primes
$\fP$ such that $N\fP>M$ gives the desired estimate \eqref{equnif} for
$\W_\fP^{(1)}$, and we have proven Theorem~\ref{thunif}.

Theorem~\ref{thsieve} now follows from Theorems~\ref{thmaincountgon}
and \ref{thunif} via a sifting argument just as in \cite[\S2.7]{BS2}.

\vspace{.1in}
\noindent {\bf Remark 5.4:} The error term of Theorem \ref{thsieve} can be
improved as follows: first, note that by repeating the arguments of
\S4, the error term of Theorem \ref{thunif} can be improved to
$O(X/(M\log M)+X^{1-1/\epsilon'_n})$, where $\epsilon'_2=2d$,
$\epsilon'_3=6d$, $\epsilon'_4=12d$, and $\epsilon'_5=40d$. Second, we
use the inclusion-exclusion sieve instead of a sifting argument to
recover Theorem \ref{thsieve} from Theorems \ref{thmaincountgon} and
\ref{thunif}.  Finally, from the improvement to Theorem
\ref{thmaincountgon} explained in the remark immediately following it,
we see that the error term of Theorem~\ref{thsieve} can be improved to
$O_Z(X^{1-1/(5dd_n)})$.

\section{Application to field counting}\label{sec:fieldscounting}

Let $F$ be a fixed global field. Suppose the characteristic of $F$ is not $2$ when $n=2$. Let $S$ be a nonempty finite set of primes of $F$ containing
$M_\infty$. Suppose $S$ contains all places above $2$ when $n=2$.  Let $\Sigma=(\Sigma_\fP)$ be an acceptable collection of local
specifications for degree-$n$ extensions of $F$, such that
$\Sigma_\fP$ for $\fP\in S$ consists of only one isomorphism class of
extension of $F_\fP$.  Thus $(\Sigma_\fP)_{\fP\in S}$ defines an
$S$-specification $\sigma$.

Let $N_{\Sigma,S}(\L_\beta,\Gamma_\beta;X)$ denote the number of
$\Gamma_\beta$-orbits of elements $v$ of $\L_\beta$ such that
$|\Delta_n(v)|_S$ is at most $X$ if $F$ is a number field and is equal
to~$X$ if $F$ is a function field and such that $\theta_{F_\fP}(v)\in \Sigma_\fP$ for every prime
$\fP\in S$, and $\theta_{\cO_\fP}(v)$ is the ring of integers of a degree-$n$ extension of
$F_\fP$ contained in $\Sigma_\fP$ for every prime $\fP\notin
S$. For any prime $\fP\notin S$, we
write $V_n(\cO_\fP)^\max_{\Sigma_\fP}$ for the subset of
$V_n(\cO_\fP)^\max$ consisting of elements whose $G(F_\fP)$-orbit
corresponds to a degree-$n$ \'{e}tale extension of $F_\fP$ contained
in $\Sigma_\fP$.

We summarize the results of the previous sections in the following
theorem.

\begin{theorem}\label{thsieve22}
With notations as above, we have
\begin{equation}\label{eq:sieve2}
N_{\Sigma,S}(\L_\beta,\Gamma_\beta;X) = \frac{\Vol_{\L_\beta}(\FF(X)\cdot v_\sigma)}{\#\Aut(\sigma)}\prod_{\fP\notin S}\frac{\Vol_\fP(V_n(\cO_\fP)^\max_{\Sigma_\fP})}{\Vol_\fP(V_n(\cO_\fP))} + o(X).
\end{equation}
\end{theorem}

Let $N_{\Sigma,S}(F,\beta;X)$ denote the number of \'{e}tale
extensions $L$ over $F$ of degree $n$ such that its normal closure has
Galois group $S_n$, its local specification is contained in $\Sigma$,
at places outside of $S$ its $S$-Steinitz class is $(\chi(\beta))$,
and the norm of the $S$-discriminant of $L$ is at most $X$ if $F$ is a
number field and equal to~$X$ if $F$ is a function field. Suppose $L$
is one such field corresponding to some $v\in
\L_\beta^\max$. Theorem~\ref{thmainparam} gives the following equality of
$\cO_S$ fractional ideals,
$$\Disc_S(L) = (\chi(\beta))^2(\Delta_n(v)).$$
The norm of the ideal $\Disc_S(L)$ is then
\begin{equation}\label{eq:Snorm}
N(\Disc_S(L)) = \prod_{\fP\notin S}|\Disc_S(L)|_\fP^{-1} =
|\chi(\beta)|^{-2}\prod_{\fP\notin S}|\Delta_n(v)|_\fP^{-1} =
|\chi(\beta)|^{-2}\prod_{\fP\in S}|\Delta_n(v)|_\fP.
\end{equation} Hence to count
fields such that the norm of the $S$-discriminant is bounded by $X$,
we may count $\Gamma_\beta$-orbits in $\L_\beta^{\max}$ such that the
$S$-norm is bounded by $|\chi(\beta)|^2X.$ In other words,
\begin{equation}\label{eq:numberoffields}
N_{\Sigma,S}(F,\beta;X) = N_{\Sigma,S}(\L_\beta,\Gamma_\beta;|\chi(\beta)|^2X).
\end{equation}

We now compute all the local volumes appearing in
\eqref{eq:sieve2}. We abbreviate $V_n,G_n$ to $V,G$.

\section{Computing the product of local volumes}\label{sec:volume}

\subsection{From volumes in $V$ to volumes in $G$: Jacobian change of variable}

Observe first that the groups and the representations that we have
been using are all defined over $\Z$. Let $\uG$ and $\uV$ denote the
corresponding group and vector space defined over $\Z$ so that $G$ and
$V$ are the base changes of $\uG$ and $\uV$, respectively, to $F$. Let
$\omega_G$ $($resp. $\omega_V)$ be a top degree left-invariant
differential form that generates the rank $1$ module of top degree left-invariant
differential forms on $\uG$ $($resp. $\uV)$. Then $\omega_G$ and $\omega_V$
are well-defined up to sign. They induce Haar measures
$\omega_{G,\fP},\omega_{G,S}$ on $G(F_\fP)$ and $G(F_{S})$ and
similarly for $V$. Since $V$ is a direct sum of copies of $\bG_a$, the
volume of $V(\cO_\fP)$ computed with respect to $\omega_{V,\fP}$ is
$1$ for every finite prime and the covolume of $\cO_S$ in $V(\cO_{S})$
is $\sqrt{D_F}^{\,\dim V}$ where $D_F$ is the absolute discriminant of
$F$ (see \cite[\S2.1.1]{Weil}). Hence, we have
\begin{equation}\label{eq:covolume}
\Vol_{\L_\beta}(\FF(X)\cdot v_\sigma) = \sqrt{D_F}^{\,-\dim V}\,|\chi(\beta)|^{-2}\int_{\FF(X)\cdot v_\sigma} \omega_{V,S}(v).
\end{equation}
We wish to the compute the volumes of $\FF(X)\cdot v_\sigma$ in
$V(F_{S})$ and of $V(\cO_\fP)^{\rm max}_{\Sigma_\fP}$ in $V(F_\fP)$ by
the volumes of $\FF(X)$ in $G(F_{S})$ and $G(\cO_\fP)$ in $G(F_\fP)$
respectively. To do this, we require a Jacobian change-of-variable
formula.

Let $v$ be the generic point of $\uV$. We have a morphism $\pi_v:\uG\rightarrow \uV$ defined over $\Z$
sending $g$ to $\pi_v(g) = g.v$. Then there exists a polynomial $J:\uV\rightarrow \GG_a$ such that
\begin{equation}\label{eqjac1}
(\pi_v^*\omega_V)(g) = J(v)\chi(g)^2\omega_G(g)
\end{equation}
 since the top degree form $(\pi_v^*\omega_V)(g)/\chi(g)^2$ is
 left-invariant on $G$ (over $\Z)$. For any $h\in \uG$, considering the map $g\mapsto gh\mapsto gh.v$
 shows that $J(h.v)=\chi(h)^2J(v)$. This shows that
 $J$ is an integer $\J$ times the relative invariant
 $\Delta_n(v)$.

We have the following proposition computing the value of $\J$.
\begin{proposition}\label{propJvalue}
Let $\J$ be as above. Then $\J=2$ when $n=2$ and $\J=\pm1$ when $n=3,4,5$.
\end{proposition}
\begin{proof}
When $n=2$, the group $\underline{G_2}=\GG_m$ acts on $\underline{V_2}=\GG_a$ via $g\cdot
v=g^2v$. We tale $\omega_G = g^{-1}dg$ and $\omega_V=dv$. It is then easy to compute $\J$ to be $2$.

When $n=3,4,5$, we use Theorem \ref{thfieldparamR} for the rings
$\Z_p$ and $\F_p$ to compute the value of $|\J|_p$ for all primes
$p$. Consider the set $Y$ of elements in $\uV(\Z_p)$ whose $\uG(\Z_p)$-orbit corresponds to the rank $n$ ring $\Z_p^n$. The set $Y$ is a single $\uG(\Z_p)$-orbit and the size of
the stabilizer in $\uG(\Z_p)$ of an element in $Y$ is $n!$. From \eqref{eqjac1}, we obtain
\begin{equation}\label{eqjp1}
\Vol(Y)=\frac{|\J|_p}{n!}\Vol(\uG(\Z_p)),
\end{equation}
where the volumes of $Y$ and $\uG(\Z_p)$ are computed with respect
to the measures induced by $\omega_V$ and $\omega_G$, respectively.

The reduction $\bar{Y}$ of $Y$ in $\uV(\F_p)$ corresponds to the ring
extension $\F_p^n$ of $\F_p$, and consists of a single
$\uG(\F_p)$-orbit. The size of the stabilizer in $\uG(\F_p)$ of an
element in $\bar{Y}$ is again $n!$. Hence we have $\#\bar{Y} =
\#\uG(\F_p)/n!$.  The set $Y$ is the inverse image of $\bar{Y}$ under
the reduction modulo $p$ map. Indeed if $v\in \uV(\Z_p)$ such that the
$\theta_{\F_p}(v \!\!\mod \fP)=\F_p^n$,
then $\Delta_n(v)$ is nonzero modulo $p$. Hence $\theta_{\Q_p}(v)$ is a product of unramified extensions of $\Q_p$,
$\theta_{\Z_p}(v)$ is the product of their rings
of integers and $\theta_{\F_p}(v\!\!\mod\fP)$ is the product of their residue fields. Since all of
these residue fields are $\F_p$ and the extensions of $\Q_p$ are
unramified, it follows that the $\theta_{\Q_p}(v)=\Q_p^n$ and $v\in Y$.  Hence we have
\begin{equation}\label{eqjp2}
\Vol(Y)=\frac{\#\bar{Y}}{p^{\dim \uV}}=\frac{1}{n!}\frac{\#\uG(\F_p)}{p^{\dim \uG}},
\end{equation}
since the dimension of $\uG$ equals the dimension of $\uV$. Finally by Hensel's lemma (cf.\ \cite[Proposition 4.7]{G}),
\begin{equation}\label{eqjp3}
\Vol(\uG(\Z_p))=\frac{\#\uG(\F_p)}{p^{\dim \uG}}.
\end{equation}
Combining equations \eqref{eqjp1}, \eqref{eqjp2}, and
\eqref{eqjp3} gives $|\J|_p=1$ for all primes $p$. Alternately, one can compute $\J$ by an explicit
computation in the three cases.
\end{proof}

We thus have the following Jacobian change of variable formula.

\begin{proposition}\label{prop:Jchange}
Suppose the characteristic of $F$ is not $2$ if $n=2$. Then there exists a constant $\J\in F^\times$ such that
for any place $\fP$
of $F$, any nonzero element $v_0$ of $V(F_\fP)$, any measurable
function $\varphi$ on $V(F_\fP)$ and any measurable subset $\FF$ of
$G(F_\fP)$, we have
\begin{equation}\label{eq:Jchange}
\int_{\FF.v_0} \varphi(v)\omega_{V,\fP}(v) = |\Delta_n(v_0)|_\fP|\J|_\fP \int_{\FF} \varphi(g.v_0)|\chi(g)|_\fP^2 \omega_{G,\fP}(g),
\end{equation}
where as before $\FF.v_0$ is viewed as a multiset so that it is in
bijection with $\FF$. The value of $\J$ is $2$ when $(G,V)=(G_2,V_2)$
and the value of $\J$ is $\pm1$ when $(G,V)=(G_n,V_n)$ for $n=3,4,5$.
\end{proposition}

We recall the definition of $\FF(X)$ from \S\ref{sec:fund}. Let
$\Omega$ denote a fundamental domain for the action of $\Gamma_\beta$
on $G(F_S)^1$ by left multiplication. When $F$ is a number field, let
$\Lambda_X$ denote the image
$\iota((0,(X/|\Delta_n(v_\sigma)|)^{1/2}])$ where $\iota$ is the
  embedding of $\R^+$ in $G(F_S)$ normalized such that
  $|\chi(\iota(r))| = r$ for every $r\in \R^{+}$. When $F$ is a
  function field, let $\Lambda_X$ denote an arbitrary element of
  $G(F_S)^{(\sigma)}_X.$ Using Proposition~\ref{prop:Jchange}, we
  obtain the following result:

\begin{equation}\label{eq:volinftyS2}
\int_{\FF(X)\cdot v_\sigma} \omega_{V,S}(v) =
\begin{cases}
\displaystyle|\Delta_n(v_\sigma)| \Bigl(\prod_{\fP\in S} |\J|_\fP \Bigr) \int_{\Lambda_X\Omega} |\chi(g)|^2 \omega_{G,S}(g) & \text{if }F\text{ is a number field},\\
\displaystyle X \Bigl(\prod_{\fP\in S} |\J|_\fP \Bigr) \int_{\Omega} \omega_{G,S}(g) & \text{if }F\text{ is a function field}.
\end{cases}
\end{equation}

For $V(\cO_\fP)^{\rm max}_{\Sigma_\fP}$ for a prime ideal $\fP$ of
$\cO_S$, we break it up into $G(\cO_\fP)$-orbits: $$V(\cO_\fP)^{\rm
  max}_{\Sigma_\fP} = \bigcup_{v_\fP\in G(\cO_\fP)\backslash
  V(\cO_\fP)^{\rm max}_{\Sigma_\fP}} G(\cO_\fP)v_\fP.$$ By
Theorem~\ref{thfieldparamR}, each orbit $G(\cO_\fP)v_\fP$ in
$V(\cO_\fP)^\max_{\Sigma_\fP}$ corresponds uniquely to an \'etale
extension $K$ over $F_\fP$ of degree $n$ with $|\Delta_n(v_\fP)|_\fP =
|\Disc(K/F_\fP)|_\fP$ and $K\in\Sigma_\fP$, so
\begin{eqnarray*}
\int_{G(\cO_\fP)v_\fP}\omega_{V,\fP}(v)&=&\frac{1}{\#\Stab_{G(\cO_\fP)}(v_\fP)}\int_{G(\cO_\fP).v_\fP}\omega_{V,\fP}(v)\\
&=&\frac{1}{\#\Stab_{G(F_\fP)}(v_\fP)} |\Delta_n(v_\fP)|_\fP|\J|_\fP\int_{G(\cO_\fP)}\omega_{G,\fP}(g)\\
&=&\frac{|\Disc(K/F_\fP)|_\fP}{\#\Aut(K/F_\fP)} |\J|_\fP \int_{G(\cO_\fP)}\omega_{G,\fP}(g).
\end{eqnarray*}
Hence adding up all the $G(\cO_\fP)$-orbits gives
\begin{equation}\label{eq:volp2}
\int_{V(\cO_\fP)^{\rm max}_{\Sigma_\fP}}\omega_{V,\fP}(v) =
\Bigl(\sum_{K\in \Sigma_\fP}
\frac{|\Disc(K/F_\fP)|_\fP}{\#\Aut(K/F_\fP)}\Bigr)|\J|_\fP
\int_{G(\cO_\fP)}\omega_{G,\fP}(g) = \frac{N\fP}{N\fP-1}m_{\fP,S}(\Sigma_\fP)|\J|_\fP\int_{G(\cO_\fP)}\omega_{G,\fP}(g),
\end{equation}
where recall $m_{\fP,S}(\Sigma_\fP)$ is defined by $$m_{\fP,S}(\Sigma_\fP) = \frac{N\fP-1}{N\fP}\sum_{K\in \Sigma_\fP}  \frac{|\Disc(K/F_\fP)|_\fP}{\#\Aut(K/F_\fP)}.$$

We now combine \eqref{eq:numberoffields}, \eqref{eq:covolume},
\eqref{eq:volinftyS2} and \eqref{eq:volp2}. Note that the local
factors $|\J|_\fP$ cancel out due to the product formula. We summarize
our computation in the following proposition.

\begin{proposition}\label{prop:VtoG}
With notations as above, we have
\begin{equation*}\label{eq:VtoG1a}
N_{\Sigma,S}(F,\beta;X) =
\displaystyle\frac{\sqrt{D_F}^{\,-\dim
    V}|\chi(\beta)|^{-2}|\Delta_n(v_\sigma)|}{\#\Aut(\sigma)}\prod_{\fP\notin
  S}\Bigl(m_{\fP,S}(\Sigma_\fP)\frac{N\fP}{N\fP-1}\int_{G(\cO_\fP)}\omega_{G,\fP}(v)\Bigr)\int_{\Lambda_{|\chi(\beta)|^2X}\Omega}
|\chi(g)|^2 \omega_{G,S}(g)+o(X)
\end{equation*}
if $F$ is a number field, and
\begin{equation*}\label{eq:VtoG1b}
N_{\Sigma,S}(F,\beta;X) =
\displaystyle\frac{\sqrt{D_F}^{\,-\dim V}X}{\#\Aut(\sigma)}\prod_{\fP\notin S}\Bigl(m_{\fP,S}(\Sigma_\fP)\frac{N\fP}{N\fP-1}\int_{G(\cO_\fP)}\omega_{G,\fP}(v)\Bigr)\int_{\Omega} \omega_{G,S}(g) +o(X)
\end{equation*}
if $F$ is a function field.
\end{proposition}

\subsection{Product of the volumes over $G$: Tamagawa number}\label{sec:product}

We now compute the product of the local volumes over $G$ using the
Tamagawa number. See Section~\ref{sec:tam} for the definition of the
Tamagawa number of a reductive group over a global field where the
(global) character group has rank $1$. The Tamagawa number of $G =
G_n$ is $1$ for $n=2,3,4,5.$ Note also that $\dim\,G = \dim\, V.$

When $F$ is a number field, we have
\begin{eqnarray*}
&\phantom{=}&\sqrt{D_F}^{\,-\dim V}|\chi(\beta)|^{-2}|\Delta_n(v_\sigma)|\Bigl(\prod_{\fP\notin S}\frac{N\fP}{N\fP-1}\int_{G(\cO_\fP)}\omega_{G,\fP}(v)\Bigr)\int_{\Lambda_{|\chi(\beta)|^2X}\Omega} |\chi(g)|^2 \omega_{G,S}(g) \\
&=& |\chi(\beta)|^{-2}|\Delta_n(v_\sigma)|
\underset{s=1}{\operatorname{Res\,}}\zeta_S(s)\,\tau_G^1\bigl(\prod_{\fP\notin S} G(\cO_\fP) \times \Omega\bigr) \int_{0}^{(|\chi(\beta)|^2X/|\Delta_n(v_\sigma)|)^{1/2}} \lambda^2 d^\times\lambda\\
&=& \frac12\,
\underset{s=1}{\operatorname{Res\,}}
\zeta_S(s)X\, \tau_G^1\bigl(\prod_{\fP\notin S} G(\cO_\fP) \times \Gamma_\beta \backslash G(F_{S})^1\bigr).
\end{eqnarray*}
When $F$ is a function field over $\F_q$, we have
$$\sqrt{D_F}^{\,-\dim V}X\Bigl(\prod_{\fP\notin S}\frac{N\fP}{N\fP-1}\int_{G(\cO_\fP)}\omega_{G,\fP}(v)\Bigr)\int_{\Omega} \omega_{G,S}(g) =\log q\,\underset{s=1}{\operatorname{Res\,}}
\zeta_S(s)X\,\tau_G^1\bigl(\prod_{\fP\notin S} G(\cO_\fP) \times \Gamma_\beta \backslash G(F_{S})^1\bigr).$$
Hence by Proposition~\ref{prop:VtoG},
\begin{equation}\label{eq:VtoGtotau}
N_{\Sigma,S}(F,\beta;X) = c\frac{\Res_{s=1}\zeta_S(s)X}{\#\Aut(\sigma)} \, \tau_G^1\bigl(\prod_{\fP\notin S} G(\cO_\fP) \times \Gamma_\beta \backslash G(F_{S})^1\bigr)\prod_{\fP\notin S}m_{\fP,S}(\Sigma_\fP) + o(X),
\end{equation}
where $c=\frac12$ when $F$ is a number field and $\log q$ when $F$ is
a global field over $\F_q$.

\section{Proof of the main theorems}\label{sec:proof}

In this section we prove Theorems \ref{thmainfield}, \ref{thmainfield2} and \ref{thmainfieldS} in the case where the characteristic of $F$ is not $2$ when $n=2$. In Theorem \ref{thmainfieldS}, we also assume that $S$ contains all the places above $2$ when $n=2$. In the next section, we handle the case of quadratic extensions in any characteristic and any $S$ using a different representation.

We prove Theorem~\ref{thmainfieldS} first. Let
$\Sigma=(\Sigma_\fP)$ be an acceptable collection of local
specifications. For every $S$-specification $\sigma$ allowed in
$\Sigma$, let $\Sigma(\sigma)$ denote the subset of $\Sigma$ where the
local specifications at places of $S$ are given by $\sigma$. Summing up
\eqref{eq:VtoGtotau} over all the $\beta$, using the
formula $$\tau_G^1(G(F)\backslash G(\A)^1) = \sum_{\beta\in\cl_S}
\tau_G^1(\prod_{\fP\notin S} G(\cO_\fP) \times \Gamma_{\beta} \backslash
G(F_{S})^1),$$ and the fact that the Tamagawa number of $G$ is $1$, we
obtain
\begin{equation}\label{eq:VtoGtotausumbeta}
\sum_{\beta\in\cl_S} N_{\Sigma(\sigma),S}(F,\beta;X) = c\,\Res_{s=1}\zeta_S(s)X\frac{1}{\#\Aut(\sigma)} \, \prod_{\fP\notin S}m_{\fP,S}(\Sigma_\fP) + o(X).
\end{equation}
Adding up the $S$-specifications $\sigma$ then gives the desired
equation \eqref{eq:combinedcount2}.

Finally we prove Theorem~\ref{thmainfield2}. Let $\Sigma=(\Sigma_\fP)$
be an acceptable collection of local specifications. Let $S$ be any
nonempty finite set of places containing $M_\infty$ and when $n=2$ and
$F$ is a number field, all the places above $2$. Fix any
$S$-specification $\sigma=(L_\fP)_{\fP\in S}$ allowed in
$\Sigma$. Define
$$\Disc(\sigma) = \prod_{\fP\in S-M_\infty} |\Disc(L_\fP/F_\fP)|_\fP.$$
Suppose $L$ is a degree-$n$ extension of $F$ whose normal closure has Galois group $S_n$ and has local specification $\sigma$ at places of $S$. Then
$$N(\Disc_S(L)) = N(\Disc(L/F))\Disc(\sigma).$$
Hence, we have
\begin{eqnarray*}
N_n(F,X) &=& \sum_\sigma \sum_{\beta\in\cl_S} N_{\Sigma(\sigma),S}(F,\beta;X\Disc(\sigma))\\
&=&c\,\Res_{s=1}\zeta_S(s)X\sum_\sigma\frac{\Disc(\sigma)}{\#\Aut(\sigma)}\, \prod_{\fP\notin S}m_{\fP,S}(\Sigma_\fP) + o(X)\\
&=&c\,\Res_{s=1}\zeta_S(s)X\prod_{\fP\in M_\infty} \sum_{K\in \Sigma_\fP}\frac{1}{\#\Aut(K/F_\fP)}\prod_{\fP\notin M_\infty} \sum_{K\in \Sigma_\fP}  \frac{|\Disc(K/F_\fP)|_\fP}{\#\Aut(K/F_\fP)}.
\end{eqnarray*}
We have now proved Theoerem
\ref{thmainfield2}. Theorem~\ref{thmainfield} then follows from the
formulae \eqref{eq:localmassfinite}, \eqref{eq:localmassR},
\eqref{eq:localmassC} for the masses.

\section{Quadratic extensions of global fields}\label{sec:quadratic}
In this section, we count quadratic extensions of global fields,
thereby finishing the proof of Theorems \ref{thmainfield},
\ref{thmainfield2}, and \ref{thmainfieldS}. The $R$-orbits of the
representation $(G_2,V_2)$ described in Section \ref{sec:param} do not
yield all quadratic extensions of a PID $R$ if $2$ is not invertible
in $R$. To obtain a proof of the main theorems in the case $n=2$ that
is uniform over global fields of all characteristics, we consider the
following representation $V$ of $G$, where for a ring $R$ we have
\begin{equation*}
\begin{array}{rcl}
G(R)&=&\Bigl\{\Bigl(\begin{array}{cc}1 & \\ n &\lambda\end{array}\Bigr):
n\in R,\lambda\in R^\times\Bigr\},\\[.15in]
V(R)&=&\{x^2+axy+by^2:a,b\in R\}.
\end{array}
\end{equation*}
The group $G$ acts on $V$ via the action $\gamma\cdot
f(x,y)=f((x,y)\cdot\gamma)$ for $\gamma\in G$ and $f\in V$. Note that
the group $G$ is not reductive. The techniques in Sections 3 through 8
do not apply generally to representations of
non-reductive groups. However, they do apply in our case.

\subsection{Parametrization of quadratic extensions}
The ring of relative invariants for the action of $G$ on $V$ is freely
generated by one element, namely, the discriminant $\Delta(f)=a^2-4b$
of the binary quadratic form $x^2+axy+by^2$. The group of characters of
$G$ is generated by the determinant. For $\gamma\in G$ and $f\in V$, we have
\begin{equation*}
\Delta(\gamma\cdot f)=\det(\gamma)^2\Delta(f).
\end{equation*}

Let $R$ be a fixed ring. Given an element $f\in V(R)$, we can construct a ring
\begin{equation*}
R_f:=R[x]/f(x,1)
\end{equation*}
that is rank-$2$ as an $R$ module. The ring $R_f$ is generated
(as a module) by the elements $1$ and $x$, and the ring structure on
$R_f$ is determined by the equation $x^2=-ax-b$. There are two
automorphisms of $R_f$ over $R$: the trivial automorphism, and
$\sigma$ defined by $\sigma(1)=1$ and $\sigma(x)=-a-x$. Note these two automorphisms are the same precisely when $F$ has characteristic $2$ and $\Delta(f)=0$. We can then
compute the discriminant of $R_f$ over $R$ to be
\begin{equation*}
\disc(R_f)=\det\Bigl(\begin{array}{cc}1 & x\\1 & -a-x\end{array}\Bigr)^2
  = a^2-4b
\end{equation*}
and verify that
\begin{equation*}
\disc(R_f)=\Delta(f).
\end{equation*}
Clearly, the ring $R_f$ is independent of the $G(R)$-orbit of $f$.

Conversely, suppose $R$ is a PID. If $R'$ is a rank-$2$ ring over $R$, then we may find a basis of the form
$\{1,\omega\}$ is for $R'$ as a module over $R$ and the ring structure of $R'$ is determined by the equation
$\omega^2=-a\omega-b$ for some $a,b\in R$. Thus $R'\cong R_f$ for $f(x,y)=x^2+axy+by^2$. The choice of $w$ is well-defined up to $w\mapsto \lambda w - n$ for any $\lambda\in R^\times$ and $n\in R$. The effect on the multiplication table is exactly as the action of $G(R)$ on $V(R)$.

We therefore have the following theorem.
\begin{theorem}\label{thparam2}
Let $R$ be a PID. There is a bijection between the set of
$G(R)$-orbits on $V(R)$ and the set of
rank-$2$ rings over $R$. For any $f\in
V(R)$, with the additional assumption that $\Delta(f)\neq0$ if $F$ has characteristic $2$, we have
\begin{equation*}
\Aut(R_f)\cong \Stab_{G(R)}(f)\cong \Z/2\Z,
\end{equation*}
where $R_f$ is the ring corresponding to $R$.
\end{theorem}
\begin{proof}
Only the last assertion requires proof. The automorphism group of
$R_f$ is isomorphic to $\Z/2\Z$, independent of $f$. The theorem
follows since the stabilizer in $G(R)$ of an element $f=x^2+axy+by^2$
in $V(R)$ consists of two elements, namely the identity and the matrix
$\Bigl(\begin{array}{cc}1 &\\a & -1\end{array}\Bigr)$ in $G(R)$.
\end{proof}

The finiteness of class numbers for any affine algebraic group of finite type over a global field is proved by Borel \cite[Theorem 5.1]{Brl} in the number field case and by Conrad \cite[Theorem 1.3.1]{Cnrd} in the function field case. For the group $G$, this finiteness follows from a much easier argument (see below). The parametrization theorem for quadratic field extensions (cf. Theorem \ref{thmainparam}) then follows formally as in Section \ref{sec:param}.

\subsection{Class group, Tamagawa measure and reduction theory}

The group $G$ fits in a split exact sequence
where $i$ sends $n$ to the matrix $\Bigl(\begin{array}{cc}1 &\\n & 1\end{array}\Bigr)$ and $j$ sends $\lambda$ to the diagonal matrix $\Bigl(\begin{array}{cc}1 &\\ & \lambda\end{array}\Bigr)$.
Hence as far as the class group, the Tamagawa measure and the reduction theory of $G$ are concerned, $G$ behaves just like the cartesian product $\bG_a\times \bG_m.$ More precisely, let $S$ be any nonempty set of places of $F$ containing $M_\infty.$ Then the maps $i:
\bG_a\rightarrow G$ and $j:\bG_m\rightarrow G$ define the following bijections of sets
\begin{eqnarray*}
\bigl(\prod_{\fP\notin S}G(\cO_\fP)\bigr)\backslash G(\A_S)/G(F) &\simeq& \bigl(\prod_{\fP\notin S}\bG_a(\cO_\fP)\bigr)\backslash \bG_a(\A_S)/\bG_a(F) \times \bigl(\prod_{\fP\notin S}\bG_m(\cO_\fP)\bigr)\backslash \bG_m(\A_S)/\bG_m(F)\\
&=&\bigl(\prod_{\fP\notin S}\bG_m(\cO_\fP)\bigr)\backslash \bG_m(\A_S)/\bG_m(F)\\
G(\A)^1 &\simeq& \bG_a(\A) \times \bG_m(\A)^1\\
G(F)\backslash G(\A)^1 &\simeq& \bG_a(F)\backslash \bG_a(\A) \times \bG_m(F)\backslash \bG_m(\A)^1\\
G(\cO_S)\backslash G(F_S)^1 &\simeq&\bG_a(\cO_S)\backslash \bG_a(F_S) \times \bG_m(\cO_S)\backslash \bG_m(F_S)^1
\end{eqnarray*}
The first bijection gives the finiteness of the class group of $G$. The second bijection defines the Tamagawa measure on $G(\A)^1$. The third bijection shows that the Tamagawa number of $G$ is $1$. Finally, the last bijection gives a fundamental domain for the action of $G(\cO_S)$ on $G(F_S)^1$. Note now once the height condition is imposed, the fundamental domain will be bounded so no cusp analysis needs to be done and the asymptotic for the number of integral orbits (cf. Theorem \ref{thmaincountgon}) follows as in Section \ref{sec:count}. We note that in general, the main difficulty with non-reductive groups is the cusp analysis, which does not pose an issue for us in this case because the group is very simple.

\subsection{Congruence conditions and a maximal sieve}

Next we impose infinitely many congruence conditions to go from
counting integral orbits to counting field extensions as in Sections
\ref{sec:sieve} and \ref{sec:fieldscounting}. Since we are working with degree-$2$ extensions, there are no extra ramifications. The needed
maximal sieve (cf. Theorem \ref{thunif}) then follows from the same method in
Section \ref{sec:sieve} that dealt with $\W_{\fP}^{(1)}$.

We note that this argument implies that if $R_f$ is non-maximal, then
after a change of variable using an element of $G(\cO_\fP)$,
$f(x,y)=x^2 + axy + by^2$ with $a\in \fP$ and $b\in\fP^2$. This is
also evident from the parametrization of quadratic rings
above. Indeed, if $R_f$ is strictly contained in a rank-$2$ ring
$R_f'$ with basis $\{1,w\}$, then there exists a positive integer $k$
such that $\{1,\pi_\fP^kw\}$ is a basis for $R_f$ where $\pi_\fP$ is a
uniformizer of $\cO_\fP$. From the multiplication table we see that
$a\in\fP^k$ and $b\in\fP^{2k}$. Note when the characteristic of $F$ is
$2$, we have $\Delta(f) = a^2$ and so $\fP^2\mid\Delta(f)$ if and only
if $\fP\mid a.$ This illustrates why squarefree-ness is not the
correct condition to use in characteristic $2$.

We now have the asymptotic for the number of quadratic extensions of
an arbitrary global field in terms of various local volumes in $V$
(cf. Theorem \ref{thsieve22} and equation \eqref{eq:numberoffields}).

\subsection{Proof of the main theorems}

We perform volume computations by pulling back to $G$ as in Section
\ref{sec:volume}. Since our parametrization theorem now works for any
PID, the proof of Proposition \ref{propJvalue} in the case $n=3,4,5$
applies to show that $\J=\pm1$. Alternatively, one may compute it
directly using the differential forms $\omega_V = da\wedge db,$
$\omega_G = \lambda^{-1}dn\wedge d^\times\lambda$ and the orbit map
given by
$$\Bigl(\begin{array}{cc}1 &\\n & \lambda\end{array}\Bigr).(a,b) =
  (\lambda a + 2n,\lambda^2 b + \lambda n a + n^2).$$ The Jacobian
  change of variable formula (cf. Proposition \ref{prop:Jchange}) then
  holds in any characteristic. We have shown above that the Tamagawa
  number of $G$ is $1$ and so we have the analogous cancelations to
  those in Section \ref{sec:product}.

Finally Theorems \ref{thmainfield}, \ref{thmainfield2}, and
\ref{thmainfieldS} follow from the same formal argument as in Section
\ref{sec:proof}.

\section{Unramified extensions of number fields}

In this section, we prove Theorems \ref{thclassgps} and
\ref{thurna}. Before turning to the proofs of these theorems, we
introduce the notion of splitting types. Let $K$ be a local field and
let $L$ be a finite etale extension of $K$. Then
$L=L_1\oplus\cdots\oplus L_k$ is a sum of local field extensions of
$K$. We say that the {\it splitting type} of $L/K$ is
$(f_1^{e_1}\ldots f_k^{e_k})$, where $f_i$ is the local degree of
$L_i/K$ and $e_i$ is the ramification degree of $L_i/K$. If $K$ is
archimedian, we set $f_i=[L_i:K]$ and $e_i=1$ for all $i$. For a
global field $F$, a finite extension $L/F$, and a place $\fP$ of $K$,
we say that the {\it splitting type} of $L/F$ at $\fP$ is the
splitting type of $L\otimes F_\fP/F_\fP$.

\subsection{Proof of Theorem \ref{thclassgps}}

Let $F$ be a fixed global field. We first compute the average size of
the $3$-torsion subgroups of the relative class groups of quadratic
extensions of $F$. Let $L$ be a quadratic extension of $F$. By
duality, the number of index-$3$ subgroups of $\Cl(L/F)$ is equal to
$(h_3(L/F)-1)/3$. Class field theory implies that the set of index-$3$
subgroups of $\Cl(L/F)$ is in bijection with the set of unramified
cubic extensions $L_3$ of $L$ such that $L_3/F$ is Galois with Galois
group $S_3$. There exists an intermediate extension $K$ between $L_3$
and $F$ with $[K:F]=3$, unique up to conjugacy. Furthermore, $K$ is
nowhere overramified over $F$, i.e., no prime ideal $\p$ of $F$ splits
as $\PP^3$. Conversely, given a nowhere overramified cubic
$S_3$-extension $K$ of $F$, let $L_3$ denote the Galois closure of $K$
over $F$. Then $L_3$ contains a unique subfield $L$ that is a
quadratic extension of $F$, and $L_3/L$ is unramified and thus
corresponds to an index-$3$ subgroup of $\Cl(L/F)$. Furthermore, the
relative discriminants of $K$ and $L$ over $F$ are the same. Given a
place $\fP$ of $F$, the splitting type of $L/F$ at $\fP$ is determined
by the splitting type of $K/F$ at $\fP$. These splitting types are
easily computed using a short and elegant note of Wood \cite{MWsp},
and we list them in the following table.

\begin{table}[ht]
  \centering
  \begin{tabular}{|c | c|}
    \hline
    Splitting type of $K$ at $\fP$ & Splitting type of $L$ at $\fP$\\
    \hline
    $(111)$&$(11)$\\
    $(3)$&$(11)$\\
    $(12)$&$(2)$\\
    $(1^21)$&$(1^2)$\\
    \hline
  \end{tabular}
\caption{Relation between cubic and quadratic splitting
  types}\label{tabcubicsp}
\end{table}

Table \ref{tabcubicsp} gives a map $\rho$ from the cubic splitting
types of $K$ to the quadratic splitting types of~$L$. Let
$\Sigma=(\Sigma_\fP)_\fP$ be an archimedeally pure acceptable
collection of local specifications for quadratic extensions of $F$. We
define the collection $\Lambda=(\Lambda_\fP)_\fP$ of local
specifications for cubic equations by setting
$\Lambda_\fP=\rho^{-1}(\Sigma_\fP)$. The acceptability of $\Sigma$
implies that $\Lambda$ is also acceptable. Furthermore, the above
discussion implies that we have
\begin{equation*}
  \displaystyle\frac{\displaystyle\sum_{L\in S_{2,\Sigma}(F,X)}(h_3(L/F)-1)/2}{\#S_{2,\Sigma}(F,X)}=
  \frac{N_{3,\Lambda}(F,X)}{2N_{2,\Sigma}(F,X)},
\end{equation*}
where the factor of $2$ in the denominator appears since
$N_{2,\Sigma}(F,X)$ counts quadratic extensions of $F$ weighted by a
factor of $1/2$. Using Theorem \ref{thmainfield2}, which computes the
asymptotics for $N_{n,\Sigma}(F,X)$, we therefore obtain
\begin{equation}\label{eqpf3tor2pf}
  \frac{\displaystyle\sum_{L\in S_{2,\Sigma}(F,X)}h_3(L/F)}{\#S_{2,\Sigma}(F,X)}=1+
  \frac{\displaystyle\prod_\fP m_{\fP}(\Lambda_\fP)}{\displaystyle\prod_\fP m_{\fP}(\Sigma_\fP)},
\end{equation}
where the quantities $m_\fP$ are defined in the statement of Theorem
\ref{thmainfield2} and the products are taken over all places $\fP$ of
$F$. For every non-archimedean place $\fP$, it is easy to see from
Table \ref{tabcubicsp} that $m_\fP(\Lambda)=m_\fP(\Sigma)$. When
$F_\fP=\C$, we have $K\otimes\F_p=\C^2$ and $L\otimes\F_\fP=\C^3$, and
it follows that $m_\fP(\Lambda_\fP)/m_\fP(\Sigma_\fP)=1/3$. If
$F_\fP=\R$, then $\Sigma_\fP$ has two choices: either
$\Sigma_\fP=\{\C\}$ which implies that $\Lambda_\fP=\{\R\times\C\}$
and $m_\fP(\Lambda_\fP)=m_\fP(\Sigma_\fP)$ or $\Sigma_\fP=\{\R^2\}$
which implies that $\Lambda_\fP=\{\R^3\}$ and
$m_\fP(\Lambda_\fP)/m_\fP(\Sigma_\fP)=1/3$. Part (a) of
Theorem~\ref{thclassgps} now follows from \eqref{eqpf3tor2pf}.

We now prove Part (b) of Theorem~\ref{thclassgps} by computing the
average size of the $2$-torsion subgroups of the relatice class groups
of cubic extensions of a global field $F$.

This time, we take $L$ to be a cubic extension of $F$. The number of
index-$2$ subgroups of $\Cl(L/F)$ is equal to $h_2(L/F)-1$. The set of
index-$2$ subgroups of $\Cl(L/F)$ is in bijection with the set of
unramified quadratic extensions $L_2$ of $L$ such that $L_2/F$ is
Galois with Galois group $S_4$. Denote the Galois closure of $L_4$ by
$F_{24}$. This yields a quartic extension $K$ of $F$ contained in
$F_{24}$, unique up to conjugacy, which is {\it nowhere overramified}
over $F$, i.e., no prime in $F$ has splitting type $(1^21^2)$,
$(2^2)$, or $(1^4)$ in $K$. Furthermore, the relative discriminants of
$L/F$ and $K/F$ are the same. Conversely, given a quartic nowhere
overramified $S_4$ extension $K$ of $F$, let $F_{24}$ denote the
Galois closure of $K$ over $F$. Let $L$ denote the cubic resolvent
field contained in $K_{24}$, and let $L_2$ denote the unique quadratic
extension of $L$ whose Galois closure over $F$ is $F_{24}$. Then
$L_2/L$ is unramified. The following table relates the splitting types
of not overramified primes in $K$ to their splitting types in $L$.

\begin{table}[ht]
  \centering
  \begin{tabular}{|c | c|}
    \hline
    Splitting type of $L$ at $\fP$ & Splitting type of $K$ at $\fP$\\
    \hline
    $(1111)$&$(111)$\\
    $(22)$&$(111)$\\
    $(112)$&$(12)$\\
    $(4)$&$(12)$\\
    $(13)$&$(3)$\\
    $(1^211)$&$(1^21)$\\
    $(1^22)$&$(1^21)$\\
    $(1^31)$&$(1^3)$\\
    \hline
  \end{tabular}
\caption{Relation between quartic and cubic splitting
  types}\label{tabquarticsp}
\end{table}

Table \ref{tabcubicsp} gives a map $\rho$ from the quartic splitting
types of $K$ to the cubic splitting types of~$L$. For an an
archimedeally pure acceptable collection of local specifications
$\Sigma=(\Sigma_\fP)_\fP$ for cubic extensions of $F$, we define the
acceptable collection $\Lambda$ and $\Lambda^+$ of local
specifications for quartic equations by setting
$\Lambda_\fP=\Lambda^+_\fP=\rho^{-1}(\Sigma_\fP)$ for nonarchimedean
$\fP$ and for $F_\fP=\C$, and setting $\Lambda_\fP=\{\R^4\}$,
$\Lambda^+_\fP=\{\R^4,\C^2\}$ when $\Sigma_\fP=\{\R^3\}$, and setting
$\Lambda_\fP=\Lambda^+_\fP=\{\R^2\oplus\C\}$ when
$\Sigma_\fP=\{\R\oplus\C\}$.  As in Part (a), applying Theorem
\ref{thmainfield2} yields the following:
\begin{equation}\label{eqpf3tor2pf1}
\begin{array}{rcl}
\displaystyle\frac{\displaystyle\sum_{L\in S_{3,\Sigma}(F,X)}h_2(L/F)}{\#S_{3,\Sigma}(F,X)}&=&1+
\displaystyle
\frac{\displaystyle\prod_\fP m_{\fP}(\Lambda_\fP)}{\displaystyle\prod_\fP m_{\fP}(\Sigma_\fP)};\\[.3in]
\displaystyle\frac{\displaystyle\sum_{L\in S_{3,\Sigma}(F,X)}h^+_2(L/F)}{\#S_{3,\Sigma}(F,X)}&=&1+
\displaystyle
\frac{\displaystyle\prod_\fP m_{\fP}(\Lambda^+_\fP)}{\displaystyle\prod_\fP m_{\fP}(\Sigma_\fP)}.
\end{array}
\end{equation}

The ratios $m_\fP(\Lambda_\fP)/m_\fP(\Sigma_\fP)$ and
$m_\fP(\Lambda^+_\fP)/m_\fP(\Sigma_\fP)$ can be easily computed using
Table \ref{tabquarticsp}, and are seen to be $1$ when $\fP$ is
nonarchimedean. When $F_\fP=\C$ the ratios are both $1/4$, when
$\Sigma_\fP=\{\R^3\}$ we have
$m_\fP(\Lambda_\fP)/m_\fP(\Sigma_\fP)=1/4$,
$m_\fP(\Lambda^+_\fP)/m_\fP(\Sigma_\fP)=1$, and when
$\Sigma_\fP=\{\R\oplus\C\}$ we have
$m_\fP(\Lambda_\fP)/m_\fP(\Sigma_\fP)=m_\fP(\Lambda_\fP)/m_\fP(\Sigma_\fP)=1/2$.
Part (b) of Theorem~\ref{thclassgps} thus follows from
\eqref{eqpf3tor2pf1}.

\subsection{Proof of Theorem \ref{thurna}}
We first prove Part (a). Fix $n\leq 5$, and let $M$ be an
$S_n$-extension of $F$ having degree-$n!$. We denote the subfield
(unique up to conjugacy) of $M$ having degree-$n$ over $F$ by $L$, and
the subfield having degree-$2$ over $F$ by $K$. We call $K$ the {\it
  quadratic resolvent} of $L$. It is known that $M$ is unramified over
$K$ at a place $\fP$ if and only if $L$ is simply ramified over $F$ at
$\fP$, i.e., the ramified terms in the splitting type at $\fP$ of
$L/F$ contain at most a $1^2$. In particular, if $F_\fP=\R$, then
$L\otimes F_\fP$ can only be $\R^n$ or $\R^{n-2}\oplus\C$, depending
on whether $K\otimes \F_\fP$ is $\R$ or $\C$, respectively. In
general, the splitting type of $L$ determines the
splitting type of
the quadratic resolvent of $L$. When $n=3$, this relation is listed in
Table \ref{tabcubicsp}. When $n=4$, this relation can be determined
from Tables \ref{tabcubicsp} and \ref{tabquarticsp} since the
quadratic resolvent of a quartic extension $L$ of $F$ is the quadratic
resolvent of the cubic resolvent of $L$. The relation between the
splitting types of quintic extensions $L$ of $F$ and their quadratic
resolvents $K$ are listed in the following table.

\begin{table}[ht]
  \centering
  \begin{tabular}{|c | c|}
    \hline
    Splitting type of $L$ at $\fP$ & Splitting type of $K$ at $\fP$\\
    \hline
    $(11111)$&$(11)$\\
    $(1112)$&$(2)$\\
    $(113)$&$(11)$\\
    $(122)$&$(11)$\\
    $(14)$&$(2)$\\
    $(23)$&$(2)$\\
    $(5)$&$(11)$\\
    $(1^2111)$&$(1^2)$\\
    $(1^212)$&$(1^2)$\\
    $(1^23)$&$(1^2)$\\
    \hline
  \end{tabular}
\caption{Relation between quintic and quadratic splitting
  types}\label{tabquinticsp}
\end{table}

Therefore, for $n=3$, $4$, and $5$, we have maps $\rho$ from
degree-$n$ splitting types to quadratic splitting types. Now, given an
archimedeally pure acceptable collection of local specifications
$\Sigma$ for quadratic extensions of $F$, we define the acceptable
collection of local specifications $\Lambda$ for degree-$n$ extensions
of $F$ by defining $\Lambda_\fP$ to be $\rho^{-1}(\Sigma_\fP)$. From
the above discussion, it follows that we have
\begin{equation*}
E_\Sigma(A_n,S_n)=\lim_{X\to\infty}\frac{N_{n,\Lambda}(F,X)}{2N_{2,\Sigma}(F,X)},
\end{equation*}
where the factor of $2$ in the denominator arises because quadratic
extensions of $F$ are counted with weight $1/2$ in
$N_{2,\Sigma}(F,X)$. From Theorem \ref{thmainfield2} we thus obtain
\begin{equation*}
E_\Sigma(A_n,S_n)=\frac{\prod_\fP m_\fP(\Lambda)}{2\prod_\fP m_\fP(\Sigma)}.
\end{equation*}
The ratio $m_\fP(\Lambda)/m_\fP(\Sigma)$ is easily seen to be one for
nonarchimedean $\fP$ from Tables \ref{tabcubicsp}, \ref{tabquarticsp},
and \ref{tabquinticsp}.  If $F_\fP=\C$, then $\Sigma_\fP=\C^2$,
$\Lambda_\fP=\C^n$, and $m_\fP(\Lambda)/m_\fP(\Sigma)=2/n!$. For
$F_\fP=\R$, there are two possible choices for $\Sigma_\fP$: if
$\Sigma_\fP=\{\R^2\}$, then $\Lambda_\fP=\R^n$ and
$m_\fP(\Lambda)/m_\fP(\Sigma)=2/n!$, while if $\Sigma_\fP=\{\C\}$,
then $\Lambda_\fP=\R^{n-2}\oplus\C$ and
$m_\fP(\Lambda)/m_\fP(\Sigma)=1/(n-2)!$. Part (a) of Theorem
\ref{thurna} therefore follows.

To prove Part (b), we proceed as follows. Fix $n=2$, $3$, or $4$. Let
$K$ be a quadratic extension of $F$. Let $L$ be a degree-$n$ extension
of $F$ such that the relative discriminant of $L$ over $F$ divides the
relative discriminant of $K$ over $F$. Then the composite field
extension $M$ of the Galois closure of $L$ over $F$ and $K$ is an
$(S_n,S_n\times C_2)$-extension of $K$ unramified at all finite
places. If we further assumed that for every place $\fP$ of $F$ such
that $F_\fP=\R$, we have $L\otimes F_\fP=\R^n$ or $\R^{n-2}\oplus\C$
depending on whether $K\otimes F_\fP=\R^2$ or $\C$, respectively, then
$M$ is an unramifie $(S_n,S_n\times C_2)$-extension of $K$. In
particular, note that every degree-$n$ extension $L$ of $F$,
satisfying these archimedean conditions, contributes to unramified
$(S_n,S_n\times C_2)$-extensions of infinitely many quadratic
extensions $K$ of $F$, namely those whose relative discriminants are
divisible by the relative discriminant of $L$ over $F$.

We fix a large integer $Y$. There are $\gg Y$ degree-$n$ simply
ramified extensions of $F$, satisfying any prescribed archimedean
conditions whose relative discriminants have norm bounded by $Y$. We
denote their relative discriminants, written with multiplicity, by
$d_1,d_2,\ldots, d_{tY}$. For sufficiently large $X$, each such
degree-$n$ field with relative discriminant $d_i$ contributes to the
count of unramified $(S_n,S_n\times C_2)$-extensions for $\gg
X/N(d_i)$ quadratic extensions of $F$ whose relative discriminant has
norm bounded by $X$, namely those quadratic fields whose discriminants
are divisible by $d_i$. Therefore, the average number of unramified
$(S_n,S_n\times C_2)$-extensions of quadratic extensions of $F$ whose
discriminants have norm bounded by $X$ is
$$
\gg \frac{1}{N(d_1)}+\frac{1}{N(d_1)}+\cdots \frac{1}{N(d_{tY})}.
$$ Since this sum diverges as $Y$ goes to infinity, Part (b) of
Theorem \ref{thurna} follows.

\appendix
\renewcommand{\thesection}{\Alph{section}}
\makeatletter
\renewcommand{\@seccntformat}[1]{%
  \ifcsname the##1\endcsname
    Appendix~\csname the##1\endcsname:~%
  \fi}
\makeatother
\renewcommand{\thesection}{\Alph{section}}
\renewcommand{\thetheorem}{\Alph{section}.\arabic{theorem}}

\section{Appendix A: Lattice-point counting over function fields}

In this appendix, we prove Proposition~\ref{davgen} in the function
field case estimating the number of lattice points in an open compact
region over function fields. We believe the results here may also be
of independent use elsewhere in applications involving
geometry-of-numbers over function fields.

Fix a smooth projective and geometrically connected algebraic curve
$\sc$ over $\F_q$ and let $F$ be its field of rational functions. The goal of this appendix is to prove Proposition \ref{davgen} relating the number of lattice points inside a nice region to
the volume of the region in the case of function fields. The main technique we use is the theory of
Fourier analysis and Poisson summation over function fields which we
recall from \cite{VF}.

For any place $v$ of $F$, choose a uniformizer $\pi_v$ and denote by
$q_v=q^{\deg v}$ the size of the residue field $k(v) =
\cO_v/\pi_v\cO_v$. Let $k_v$ denote the canonical exponent of $v$ and
let $\chi_v$ be the (continuous) additive character $F_v\rightarrow
\C^\times$ defined in \cite[\S 2.1.2]{VF}. The divisor class of
$\sum_v k_v[v]$ is the canonical class of $\sc$. It will not be
important to us what the precise definitions of $k_v$ and $\chi_v$
are. Note that since $F_v$ is a locally compact group, the image of
$\chi_v$ lands inside the unit circle. Let $d_vx$ denote the Haar
measure on $F_v$ normalized so that the measure of $\cO_v$ is
$q_v^{-k_v/2}.$

Let $S$ be a nonempty finite set of places of $F$. Write $F_S$ for the product
$\prod_{v\in S}F_v$; $\chi_S$ for the character $\prod_{v\in S}\chi_v$
on $F_S$; and $d_Sx$ for the measure $\prod_{v\in S}d_vx_v$ on
$F_S$. The Fourier transform of an integrable continuous function $f$
with respect to $S$ is defined to be
\begin{equation}\label{eq:defFourier}
 \FF_S f(y) = \int_{F_S} f(x)\chi_S(xy) d_Sx\quad\forall y\in F_S.
\end{equation}
In stark contrast to the archimedean case, the following lemma implies
that the Fourier transform takes a continuous function with compact
support to a continuous function with compact support.

\begin{lemma}\label{lem:FourCC}
 {\rm $(\cite[\S2.1.3]{VF})$} For any subset $B$ of $F_S$, denote by
 $\chi_B$ its characteristic function. Let $n_v$ be fixed integers for
 $v\in S$. Then
\begin{equation}\label{eq:FourCC}
 \FF_S \chi_{\prod_{v\in S} \pi_v^{n_v}\cO_v} = \left(\prod_{v\in S} q_v^{-k_v/2}\right) \chi_{\prod_{v\in S}\pi_v^{-n_v-k_v}\cO_v}.
\end{equation}
\end{lemma}

For any continuous function $f$ on $F_S$ with compact support, define
its \emph{conductor} $c(f)$ to be the biggest open neighborhood of $0$
in $F_S$ such that $f(x + y) = f(x)$ for any $x\in F_S$ and $y\in
c(f).$ Then Lemma~\ref{lem:FourCC} implies that
\begin{equation}\label{eq:support}
 c(f) = \prod_{v\in S}\pi_v^{c_v}\cO_v \Rightarrow \textrm{Supp}(\FF_Sf) \subset \prod_{v\in S}\pi_v^{-c_v-k_v}\cO_v.
\end{equation}

\begin{theorem}{\rm (Poisson Summation~$\cite[{\rm Lemma}~3.5.9]{VF}$} Define
\begin{eqnarray*}
 \cO_S &=& \{x\in F:v(x)\geq0,\forall v\notin S\},\\
 \cO_S^\perp &=& \{y\in F: v(y) \geq -k_v, \forall v\notin S\},
\end{eqnarray*}
where $v(x),v(y)$ denote the valuations of $x,y$ with respect to $v$
respectively. Then for any continuous function $f$ on $F_S$ with
continuous Fourier transform, we have
\begin{equation}\label{eq:poisson}
 \sum_{x\in \cO_S} f(x) = \prod_{v\notin S} q_v^{-k_v/2}\sum_{y\in \cO_S^\perp} \FF_Sf(y).
\end{equation}
\end{theorem}

For any place $v$, the $v$-adic norm $|\alpha|_v$ of some $\alpha\in
F_v$ is defined to be $|\alpha|_v=q_v^{-v(\alpha)}.$ (This is the same
definition as given in Section \ref{sec:notation}.) For any $t=(t_v)\in F_S$,
its $S$-norm is defined to be
$$|t|_S = \prod_{v\in S}|t|_v = q^{-\sum_{v\in S}v(t)\deg v},$$
where we have abbreviated $v(t_v)$ and $|t_v|_v$ to $v(t)$ and $|t|_v$ respectively. We
are interested in the number of lattice points in a homogeneously
expanding region in $F_S$.

\begin{proposition}\label{prop:homoexp}
 Let $B$ be an open compact subset of $F_S$. Let $c$ be a positive real constant. Then for any $t=(t_v)\in F_S$ such that $|t|_v^{1/\deg v}/|t|_{v'}^{1/\deg v'}<c$ for every $v,v'\in S$,
\begin{equation}\label{eq:homoexp}
 \#\{tB \cap \cO_S\} =  \prod_{v\notin S} q_v^{-k_v/2}\Vol(tB) +O(1) = \Vol_{\cO_S}(tB) + O(1)
\end{equation}
where $\Vol$
denotes the volume computed with respect to $d_S$, and $\Vol_{\cO_S}$
is a constant multiple of $\Vol$ such that $F_S/\cO_S$ has volume $1$. The implied constant in the second summand depends only on $F$, $B$ and $c$. The dependency on $B$ is through a bound $M$ on the $S$-norm of elements of $B$ and the conductor $c(\chi_B)$ of the characteristic function of $B$.
\end{proposition}

\begin{proof}
Since $B$ is bounded, there exists a constant $M$ such that for any $b\in B$, $|b|_S<M.$ Suppose first that $|t|_S\leq1/M$. Then for any $b\in B$, $|tb|_S<1$ and so
$$\sum_{v\in S}v(tb)\deg(v) > 0.$$
On the other hand, if $b'\in \cO_S$, then $v(b')\geq0$ for every $v\notin S$ and $\sum_{v} v(b')\deg(v) = 0$. Hence we see that $tB\cap \cO_S = \emptyset$ while $tB$ is contained in the bounded open ball defined by $|b'|_S<1$. In this case, both sides of \eqref{eq:homoexp} are bounded by an absolute constant.

Suppose from now on that $|t|_S > 1/M$. Then the sum $\sum_{v\in S} -v(t)\deg(v)$ is bounded below by $-\log_q M$. By assumption on the proximity of $\log |t|_v$ for different $v$, we see that each individual $-v(t)$, for $v\in S$, is bounded below by some constant $C_1$ depending only on $M$ and $c$. Write the conductor of $\chi_B$ as $\prod_{v\in S}\pi_v^{c_v}\cO_v$
 and so $c(\chi_{tB}) = \prod_{v\in S}\pi_v^{c_v+v(t)}\cO_v.$ Hence by
 \eqref{eq:support},
$$\textrm{Supp}(\FF_S\chi_{tB})\subset \prod_{v\in S} \pi_v^{-c_v-v(t)-k_v}\cO_v.$$
If $y\in \textrm{Supp}(\FF_S\chi_{tB})\cap \cO_S^\perp$ is nonzero, then we must have
\begin{eqnarray*}
 v(y) &\geq& -k_v,\quad\forall v\notin S,\\
 v(y) &\geq& -c_v-v(t)-k_v,\quad\forall v\in S.
\end{eqnarray*}
Adding these up over all $v$ gives
$$0 = \sum_v v(y) \deg v \geq -\sum_v k_v\deg v  - \sum_{v\in S}c_v\deg v  - \sum_{v\in S}v(t)\deg v.$$
Hence the sum $\sum_{v\in S}-v(t)\deg(v)$ is also bounded above by some constant $C_2$ depending only on $F$ and $c(\chi_B)$. We write $S(t)$ for the $|S|$-tuple $(v(t))\in\Z^{|S|}.$

Suppose now that
$\sum_{v\in S}-v(t)\deg(v)\leq C_2.$
Then $S(t)$ belongs to some finite subset of $\Z^{|S|}$ depending only on $F$, $M$, $c(\chi_B)$ and $c$. If $y\in \textrm{Supp}(\FF_S\chi_{tB})\cap \cO_S^\perp$ is nonzero, then $y$ is a global section of some line bundle on $\sc$ that depends only on $S(t)$. Hence $\textrm{Supp}(\FF_S\chi_{tB})\cap \cO_S^\perp$ is contained in a finite subset $U$ of $\cO_S^\perp$ depending only on $F$, $M$, $c(\chi_B)$ and $c$. Write $B$ as a finite disjoint union of translates $B = \cup_b b+c(B)$ of its conductor. Note the number of translates needed to cover $B$ is bounded above in terms of $M$ and $c(\chi_B)$ by taking the volume of $B$. Then $tB$ is a finite union of translates $tb + c(tB)$. Hence the volume of $tB$ is bounded by a constant depending only on $F$, $M$, $c(\chi_B)$ and $c$. By the translation property of the Fourier transform and Lemma \ref{lem:FourCC}, we see that $|\FF_S\chi_{tb + c(tB)}(y)|$ for any $y\in \cO_S^\perp$ depends only on $c(\chi_B)$, $S(t)$ and $y$. Finally we apply Poisson summation \eqref{eq:poisson} to the continuous function $\chi_{tB}$. The sum on the right hand side of \eqref{eq:poisson} is bounded above by a finite sum
$$\sum_b\sum_{y\in U} |\FF_S\chi_{tb + c(tB)}(y)|$$
which depends only on $F$, $M$, $c(\chi_B)$ and $c$.

Finally suppose that $\sum_{v\in S}-v(t)\deg(v)> C_2.$ Then $\textrm{Supp}(\FF_S\chi_{tB})\cap \cO_S^\perp$ contains no nonzero element. Therefore,
$$\#\{tB \cap \cO_S\} = \sum_{x\in \cO_S}\chi_{tB}(x) = \prod_{v\notin S} q_v^{-k_v/2}\FF_S\chi_{tB}(0) = \prod_{v\notin S} q_v^{-k_v/2}\Vol(tB).$$
The second equality in \eqref{eq:homoexp} follows because $\Vol(F_S/\cO_S) = \prod_{v\notin S} q_v^{-k_v/2}$ (\cite[Proposition~3.5.1]{VF}). Note there is no error term in this case.
\end{proof}

We remark that the condition $|t|_v^{1/\deg v}/|t|_{v'}^{1/\deg v'}<c$ is not necessary. All we need in the proof is that a lower bound on $\sum_{v\in S}-v(t)\deg v$ implies a lower bound for each $-v(t)$. For example, the condition $\log |t|_v/\log |t|_{v'}<c$ is also good enough.

For our applications, we need a more general version of
Proposition~\ref{prop:homoexp} where $B$ is an open compact subset of $F_S^n$ and is first hit by some linear transform $g$ belonging to some open compact subset of $\GL_n(F_S)$ and then scaled by some diagonal matrix $t$.

\begin{proposition}\label{prop:skewexp}
 Let $B$ be an open compact subset of $F_S^n$. Let $K$ be an open compact subset of $\GL_n(F_S)$. Let $c$ be a positive real constant. Then for any $g\in K$ and for any $t=
 \diag(t_1,\ldots,t_n)\in\GL_n(F_S)$ such that for any $i=1,\ldots,n$, we have $|t_i|_v^{1/\deg v}/|t_i|_{v'}^{1/\deg v'}<c$ for any $v,v'\in S$,
 \begin{equation}\label{eq:skew}
 \#\{tgB \cap \cO_S^n\} = \Vol_{\cO_S^n}(tgB) + O(\Vol(\proj(tgB))),
 \end{equation}
 where $\Vol_{\cO_S^n}$ is a constant multiple of $\Vol$ such that
$F_S^n/\cO_S^n$ has volume $1$, and $\Vol(\proj(tgB))$ denotes the greatest $\ell$-dimensional volume of any projection of $tgB$ onto a coordinate subspace obtained by equating $n-\ell$ coordinates to zero, where $\ell$ takes all values from $1$ to $n-1$. The implied constant in the second summand depends only on $F$, $B$, $K$ and $c$.
\end{proposition}

\begin{proof}
 The definition of the Fourier transform \eqref{eq:defFourier}
 generalizes to higher dimensions where one replaces $xy$ by
 $x_1y_1+\cdots+x_ny_n$ when $x=(x_1,\ldots,x_n)$, and
 $y=(y_1,\ldots,y_n)$. The Poisson summation formula
 \eqref{eq:poisson} generalizes to
\begin{equation}\label{eq:poissonhigh}
 \sum_{x\in \cO_S^n} f(x) = \left(\prod_{v\notin S} q_v^{-k_v/2}\right)^n\sum_{y\in (\cO_S^\perp)^n} \FF_Sf(y),
\end{equation}
for any continuous function $f$ with continuous Fourier transform. The
notion of conductor also generalizes and we write
$$c(B) = \prod_{v\in S} (\pi_v^{c_{v,1}}\cO_v \times \cdots\times
\pi_v^{c_{v,n}}\cO_v)\subset\prod_{v\in S}K_v^n.$$ As $g$ varies in
the open compact set $K$, the conductor $c(\chi_{gB})$ of the
characteristic function of $gB$ takes only finitely many possibilities
and the number of translates of $c(\chi{gB})$ needed to cover $gB$
also has finitely many possibilities. Moreover, there is a real
constant $M$ that bounds the $S$-norm of every $\cO_S$-coordinate of
$gb$ for any $g\in K$ and any $b\in B$. These are the only dependency
on $K$ and $B$. In what follows, we replace $gB$ by~$B$.

As in the proof of Proposition \ref{prop:homoexp}, we may assume that
each $-v(t_i)$ is bounded below by some constant $C_1$ depending only
on $M$ and $c$. For each $i=1,\ldots,n$, we define the
constants $$C_{2,i} = \sum_v k_v\deg v + \sum_{v\in S}c_{v,i}\deg v.$$
Let $I$ be a subset of $\{1,\ldots,n\}$. Suppose
\begin{eqnarray*}
\sum_{v\in S} -v(t_i)\deg(v) &>& C_2,\quad\forall i\in I\\
\sum_{v\in S} -v(t_i)\deg(v) &\leq& C_2,\quad\forall i\notin I.
\end{eqnarray*}
Then if $(y_1,\ldots,y_n)\in \textrm{Supp}(\FF_S \chi_{tB})\cap
(\cO_S^\perp)^n$, we have $y_i=0$ for all $i\in I$. When
$I=\{1,\ldots,n\}$, then $\textrm{Supp}(\FF_S\chi_{tB})\cap
(\cO_S^\perp)^n$ contains no nonzero element and just as in the proof
of Proposition \ref{prop:homoexp}, we get the main term
$\Vol_{\cO_S^n}(tB)$ in \eqref{eq:skew}.

Suppose now $I$ is a proper subset of $\{1,\ldots,n\}$. Without loss
of generality, we may assume $I = \{i,\ldots,n\}$ for some $i\geq
2$. We write $\proj_{i,\ldots,n}$ for the projection onto the last
$n-i+1$ coordinates and $\proj_{1,\ldots,i-1}$ for the projection on
the first $i-1$ coordinates. Our goal is to show that
$$\#\{tB\cap \cO_S^n\} \leq C_3 \Vol(\proj_{i,\ldots,n}(tB))$$ for
some constant $C_3$ depending only on $F$, the bound of the $S$-norm
of elements of $B$, the conductor $c(B)$ of $B$, and $c$.

Any element $(y_1,\ldots,y_n)$ of $\textrm{Supp}(\FF_S\chi_{tB})\cap
(\cO_S^\perp)^n$ has the property that $y_j = 0$ for any
$j=i,\ldots,n$, and for any $j=1,\ldots,i-1$,
\begin{eqnarray*}
 v(y_j)&\geq& -k_v, \quad\forall v\notin S,\\
 v(y_j)&\geq& -c_{v,j} - v(t_j) - k_v \quad\forall v\in S.
\end{eqnarray*}
As before, all the possible $(i-1)$-tuples $(y_1,\ldots,y_{i-1})$, as
$t$ varies, belong to some finite set $U$ depending only on $F$, $C_1$
and $C_2$.

For any $(x_i,\ldots,x_n)\in \proj_{i,\ldots,n}(tB)$, let
$B(t,x_i,\ldots,x_n)$ denote the (open compact) subset of $F_S^{i-1}$
consisting of elements $(x_1,\ldots,x_{i-1})$ such that
$(x_1,\ldots,x_n)\in tB.$ Then
$\textrm{Supp}(\FF_S\chi_{B(t,x_i,\ldots,x_n)})$ is contained in
$U$. This follows because the conductor of $B(t,x_i,\ldots,x_n)$
contains $\proj_{1,\ldots,i-1}(c(tB))$. Let $C_S$ denote the constant
$\prod_{v\notin S} q_v^{-k_v/2}$. Then
\begin{eqnarray*}
 \#\{tB \cap \cO_S^n\} &=& C_S^n \sum_{y\in U}\FF_S\chi_{tB}(y,0,\ldots,0)\\ &=& C_S^n \sum_{y\in
   U}\int_{\substack{(x_i,\ldots,x_n)\\ \in
     \proj_{i,\ldots,n}(tB)}}\int_{\substack{(x_1,\ldots,x_{i-1})\\ \in
     B(t,x_i,\ldots,x_n)}}\chi_{S}(x_1y_1+\cdots+x_{i-1}y_{i-1})d_Sx\\ &=&
 C_S^n\int_{\substack{(x_i,\ldots,x_n)\\ \in
     \proj_{i,\ldots,n}(tB)}}\sum_{y\in U}\FF_S\chi_{B(t,x_i,\ldots,x_n)}(y)\\ &=&
 C_S^{n-i+1}\int_{\substack{(x_i,\ldots,x_n)\\ \in
     \proj_{i,\ldots,n}(tB)}}\#\{B(t,x_i,\ldots,x_n)\cap
 \cO_S^{i-1}\}d_Sx\\ &\leq&
 C_S^{n-i+1}\int_{\substack{(x_i,\ldots,x_n)\\ \in
     \proj_{i,\ldots,n}(tB)}}\#\{\proj_{1,\ldots,i-1}(tB)\cap
 \cO_S^{i-1}\}d_Sx\\ &=&
 C_S^{n-i+1}\Vol(\proj_{i,\ldots,n}(tB))\#\{\proj_{1,\ldots,i-1}(tB)\cap
 \cO_S^{i-1}\}.
\end{eqnarray*}
Finally the same argument as in the proof of Proposition \ref{prop:homoexp} shows that
$$\#\{\proj_{1,\ldots,i-1}(tB)\cap \cO_S^{i-1}\}\leq C_4$$ for some
constant $C_4$ depending only on $F$, the bound of the $S$-norm of
elements of $B$, the conductor $c(B)$ of $B$, and $c$.  This completes
the proof of Proposition~\ref{prop:skewexp}.
\end{proof}

\section*{Acknowledgments}
The first and third authors were supported by a Simons Investigator
Grant, and we thank the Simons Foundation for their kind support of
this work. It is a pleasure to thank Alireza Golsefidy, M.~S.~Raghunathan, and Jacob Tsimerman for helpful conversations.


\begin{thebibliography}{40}

\bibitem{Artn} E.\ Artin, Questions de base minimale dans la th\'{e}orie des nombres alg\'ebriques, {\it Collected papers of Emil Artin}, Addison-Wesley, Reading, MA, 1965, 229--321.

\bibitem{BII} M.\ Bhargava, Higher composition laws.\ II.\ On cubic analogues of Gauss composition, {\it Ann.\ of Math.\ $($2$)$} {\bf 159}(2004), no.\ 2, 865--886.

\bibitem{BIII} M.\ Bhargava, Higher composition laws.\ III.\ The parametrization of quartic rings, {\it Ann.\ of Math.\ $($2$)$} {\bf 159}(2004), no.\ 3,  1329--1360.

\bibitem{BIV} M.\ Bhargava, Higher composition laws.\ IV.\ The parametrization of quintic rings, {\it Ann.\ of Math.\ $($2$)$} {\bf 167}(2008), no.\ 1,  53--94.

\bibitem{Bmass} M.\ Bhargava, Mass formulae for extensions of local fields, and conjectures on the density of number field discriminants, {\it Int.\ Math.\ Res.\ Not.} no.\ 17 (2007), 20 pp.

\bibitem{hasse}
M.\ Bhargava,  A positive proportion of plane cubics fail the Hasse principle,
\url{http://arxiv.org/abs/1402.1131}.

\bibitem{B}
M.\ Bhargava,  Most hyperelliptic curves over $\Q$ have no rational points,
\url{http://arxiv.org/abs/1308.0395}.

\bibitem{dodqf} M.\ Bhargava, The density of discriminants of quartic
  rings and fields,  {\it Ann.\ of Math.\ $($2$)$} {\bf 162} (2005), no.\ 2,  1031--1063.

\bibitem{dodpf} M.\ Bhargava, The density of discriminants of quintic
  rings and fields, {\it Ann.\ of Math.\ $($2$)$} {\bf 172} (2010), no.\ 3,  1559--1591.



\bibitem{Bgeosieve} M.\ Bhargava, The geometric sieve and the density
  of squarefree values of polynomial discriminants and other invariant
  polynomials, \url{http://arxiv.org/abs/1402.0031}.







\bibitem{BG1}
M.\ Bhargava and B.\ Gross,
The average size of the $2$-Selmer group of the Jacobians of hyperelliptic curves with
a rational Weierstrass point, {\it Automorphic Representations and $L$-functions}, {\it TIFR Studies
in Math.} {\bf 22} (2013), 23--91.

\bibitem{BGW} M.\ Bhargava, B.\ Gross, and X.\ Wang, A positive
  proportion of locally soluble hyperelliptic curves over Q have no
  rational points over any odd degree extension, {\it
    J.\ Amer.\ Math.\ Soc.\ } {\bf 30} (2017), 451--493.


\bibitem{BH} M.\ Bhargava and W.\ Ho, Coregular spaces and genus one curves, {\it Camb.\ J.\ Math.} {\bf 4} (2016), no.\ 1, 1--119.

\bibitem{BH1} M.\ Bhargava and W.\ Ho, On average sizes of Selmer groups and ranks in families of elliptic curves having marked points, \url{https://arxiv.org/abs/2207.03309} 

\bibitem{BS2} M.\ Bhargava and A.\ Shankar, Binary quartic forms
  having bounded invariants, and the boundedness of the average rank
  of elliptic curves, {\it Ann.\ of Math.\ (2)} {\bf 181} (2015), no.\ 1, 191--242.

\bibitem{TC} M.\ Bhargava and A.\ Shankar, Ternary cubic forms having
  bounded invariants and the existence of a positive proportion of
  elliptic curves having rank 0, {\it Ann.\ of Math.\ (2)} {\bf 181} (2015), no.\ 2, 587--621.

\bibitem{foursel} M.\ Bhargava and A.\ Shankar, The average number of
  elements in the $4$-Selmer groups of elliptic curves is $7$,
  \url{http://arxiv.org/abs/1312.7333}\,.

\bibitem{fivesel} M.\ Bhargava and A.\ Shankar, The average size of the
$5$-Selmer group of elliptic curves is $6$, and the average rank is less than 1,
\url{http://arxiv.org/abs/1312.7859}\,.

\bibitem{BST} M.\ Bhargava, A.\ Shankar and J.\ Tsimerman, On the Davenport-Heilbronn theorems and second order terms, {\it Invent.\ Math.} {\bf 193} (2013), no.\ 2, 439--499.

\bibitem{BSW} M.\ Bhargava, A.\ Shankar and X.\ Wang, Geometry-of-numbers methods over global fields II: Coregular representations, in preparation.

\bibitem{rankone}
M.\ Bhargava and C.\ Skinner, A positive proportion of elliptic curves over $\Q$ have rank one, {\it J.\ Ramanujan Math.\ Soc.\ } {\bf 29} (2014), no.\ 2, 221--242.

\bibitem{bsddensity}
M.\ Bhargava, C.\ Skinner, and W.\ Zhang, A majority of elliptic curves over $\Q$ satisfy the Birch and Swinnerton-Dyer Conjecture, \url{http://arxiv.org/abs/1407.1826}.

\bibitem{BV}
M.\ Bhargava and I.\ Varma, The mean number of 3-torsion elements in the class groups and ideal groups of quadratic orders, {\it Proc.\ Lond.\ Math.\ Soc.\ $($2$)$} {\bf 112} (2016), no.\ 2, 235--266.

\bibitem{BV2}
M.\ Bhargava and I.\ Varma, On the mean number of 2-torsion elements in the class groups, narrow class groups, and ideal groups of cubic orders and fields, {Duke Math.\ J.\ } {\bf 10} (2015), 1911--1933.

\bibitem{Brl}
A.\ Borel, Some finiteness properties of adele groups over number fields, {\it Publ. Math. IHES} {\bf 16} (1983), 5--30.

\bibitem{BP}
A.\ Borel, G.\ Prasad, Finiteness thoerems for discrete subgroups of bounded covolume in semi-simple groups, {\it Publ. Math. IHES} {\bf 69} (1989), 119--171.

\bibitem{Cnrd}
B.\ Conrad, Finiteness theorems for algebraic groups over function fields, {\it Compos. Math.} {\bf 148} no.\ 2 (2012), 555--639.

\bibitem{DW} B.\ Datskovsky and D.\ Wright, Density of discriminants of cubic extensions, {\it J.\ Reine Angew.\ } {\bf 386} (1998), 116--138.

\bibitem{Lipschitz}
H.\ Davenport, On a principle of Lipschitz,
{\it J.\ London Math.\ Soc.} {\bf 26} (1951), 179--183.
Corrigendum: ``On a principle of Lipschitz '',  {\it J.\ London Math.\
  Soc.} {\bf 39} (1964), 580.

\bibitem{DH} H.\ Davenport and H.\ Heilbronn, On the density of discriminants of cubic fields, II, {\it Proc.\ Roy.\ Soc.\ London Ser.\ A} {\bf 322} no.\ 1551 (1971), 405--420.

\bibitem{DF}
B.\ N.\  Delone and D.\ K.\ Faddeev, {\it The theory of irrationalities
of the third degree}, AMS Translations of Mathematical Monographs
{\bf 10}, 1964.

\bibitem{AWD}
A.\ W.\ Deng, Rational points on weighted projective spaces, \url{http://arxiv.org/abs/math/9812082}.

\bibitem{G}
B.\ Gross, On the motive of a reductive group, {\it Invent.\ math.\ } {\bf 130} (1997), no.\ 2, 287--313.

\bibitem{JP} M.\ Jarden and G.\ Prasad, Appendix on the discriminant quotient formula for global field, {\it Publ. Math. IHES} {\bf 69} (1989), 115--117.

\bibitem{KW} A.\ Kable and D.\ Wright, Uniform distribution of the Steinitz invariants of quadratic and cubic extensions, {\it Compos.\ Math.} {\bf 142} no.\ 1 (2006), 84--100.

\bibitem{KL} Z.\ Klagsbrun, Davenport--Heilbronn Theorems for Quotients of Class Groups, \url{http://arxiv.org/abs/1701.02834}.

\bibitem{Mlnr}
J.\ Milnor, {\it Introduction to Algebraic $K$-theory}, Princeton University Press and University of Tokyo Press, 1971.

\bibitem{Os} J.\ Oesterle, Nombres de Tamagawa et groupes unipotents en caract\'eristique $p$, {\it Invent.\ Math.} {\bf 78} (1984), no.\ 1, 13--88.

\bibitem{PS2}
B.\ Poonen, M.\ Stoll,
Most odd degree hyperelliptic curves have only one rational point,
{\it Ann.\ of Math.\ $($2$)$} {\bf 180} (2014), no.\ 3, 1137--1166.

\bibitem{SK}
M.\ Sato and T.\ Kimura, A classification of irreducible
prehomogeneous vector spaces and their relative invariants,
{\it Nagoya Math.\ J.} {\bf 65} (1977), 1--155.

\bibitem{Smass} J.\-P.\ Serre, Une ``formule de masse'' pour les extensions totalement ramifi\'es de degr\'e donn\'e d'un corps local, {\it C.\ R.\ Acad.\ Sci.\ Paris S\'er.\ A-B} {\bf 286} no.\ 22 (1978), A1031--A1036.

\bibitem{SW}
A.\ Shankar, and X.\ Wang, Rational points on hyperelliptic curves having a
marked non-Weierstrass point, {\it Compos.\ Math.} {\bf 154} (2018), 188--222.

\bibitem{ST} A.\ Shankar. The average rank of elliptic curves over
  number fields, Ph.D thesis, Princeton 2012.

\bibitem{AJ} A.\ Shankar and J.\ Tsimerman, Counting $S_5$-fields with
  a power saving error term, {\it Forum Math.\ Sigma} \textbf{2}
  (2014), e13.

  
\bibitem{Sp} T.\ A.\ Springer, Reduction theory over global fields,
  {\it Proc.\ Indian Acad.\ Sci.\ (Math.\ Sci.\ )} {\bf 104} (1994),
  207--216.


\bibitem{Th}
J.\ Thorne, $E_6$ and the arithmetic of a family of non-hyperelliptic curves of genus 3, {\it Forum Math.\ Pi}, Vol.\ 3 (2015), e1.

\bibitem{VF} M.\ van Frankenhuijsen, The Riemann hypothesis for
  function fields, {\it London Mathematics Society Student Texts}
  \textbf{80}, Cambridge University Press 2014.

\bibitem{Weil} A.\ Weil, {\it Adeles and algebraic groups}, Birkh\"auser, 1982.

\bibitem{WoodFinite} M.\ M.\ Wood, The Distribution of the Number of Points on Trigonal Curves over $\F_q$, {\it Int. Math. Res. Not.} {\bf 23} (2012), 5444--5456

\bibitem{WoodThesis} M.\ M.\ Wood, Moduli spaces for rings and ideals, Ph.D.\ Thesis, Princeton University, June 2009.

  \bibitem{WoodQuartic} M.\ M.\ Wood, Parametrizing quartic algebras
    over an arbitrary base, {\it Algebr.\ Number Theory} {\bf
      5-8} (2011), 1069--1094.

\bibitem{MWsp} M.\ M.\ Wood, How to determine the splitting type of a prime, unpublished note. 

\bibitem{WY}
D.\ J.\ Wright and A.\ Yukie, Prehomogeneous vector spaces and field
extensions, {\it Invent.\ Math.} {\bf 110} (1992), 283--314.

\end{thebibliography}
\end{document}